\documentclass[10pt]{amsart}
\usepackage{amsfonts}
\usepackage{amssymb}
\usepackage{amsmath}
\usepackage[arrow,matrix,curve]{xy}

\newcommand{\C}{{\mathbb C}}
\newcommand{\R}{{\mathbb R}}
\newcommand{\Z}{{\mathbb Z}}
\newcommand{\N}{{\mathbb N}}

\newcommand{\E}{{\mathbb E}}

\newcommand{\K}{{\mathbb K}}

\newtheorem{theorem}{Theorem}[section]
\newtheorem{lemma}[theorem]{Lemma}
\newtheorem{remark}[theorem]{Remark}

\newtheorem{corollary}[theorem]{Corollary}
\newtheorem{proposition}[theorem]{Proposition}
\newtheorem{definition}[theorem]{Definition}

\newtheorem{conjecture}[theorem]{Conjecture}

\def\cal{\mathcal}

\newcommand{\calh}[0]{{\cal H}}

\newcommand{\calf}[0]{{\cal F}}
\newcommand{\call}[0]{{\cal L}}
\newcommand{\calg}[0]{{\cal G}}

\newcommand{\cala}[0]{{\cal A}}
\newcommand{\cale}[0]{{\cal E}}
\newcommand{\calk}[0]{{\cal K}}

\newcommand{\calp}[0]{{\cal P}}
\newcommand{\cali}[0]{{\cal I}}
\newcommand{\cals}[0]{{\cal S}}
\newcommand{\calm}[0]{{\cal M}}

\newcommand{\calc}[0]{{\cal C}}

\newcommand{\calt}[0]{{\cal T}}
\newcommand{\calj}[0]{{\cal J}}

\begin{document}

\title[Attempts to define a Baum--Connes map]{Attempts to define a
Baum--Connes map via localization of categories for inverse semigroups}
\author[B. Burgstaller]{Bernhard Burgstaller}
\address{Departamento de Matematica, Universidade Federal de Santa Catarina, CEP 88.040-900 Florian\'opolis-SC, Brasil}
\email{bernhardburgstaller@yahoo.de}

\begin{abstract}
Meyer and Nest showed that the Baum--Connes map is equivalent to a map
on $K$-theory of two different crossed products. 
This approach is strongly categorial in method
since its bases is to regard Kasparov's theory $KK^G$ as a triangulated category.
We have tried to translate this approach to the realm of inverse semigroup equivariant $C^*$-algebras 
but can prove the existence of a Baum--Connes map only under some unverified additional assumptions which we
however strongly motivate. 
Some of our results may be of independent interest,
for example Bott periodicity, the definition of induction functors, the definition of a completely novel compatible $L^2(G)$-space,
a Cuntz picture of $KK^G$, and the verification that $KK^G$ is a triangulated category.
%
%
%
\end{abstract}

\maketitle



\section{Introduction}

In \cite{meyernest}, Meyer and Nest found an alternative description
of the Baum--Connes map
$\lim_{Y \subseteq \underline{E} G} KK (C_0(Y),A) \rightarrow K(A \rtimes_r G)$
with coefficients \cite{baumconneshigson1994}, 
where $G$ denotes a locally compact, second countable group and $A$ a $G$-algebra.
(It was even achieved for groupoids of the form $G \ltimes X$.)
Fundamental for this approach is a work by Chabert and Echterhoff \cite{chabertechterhoff2001},
and the nontrivial ``observation" that Kasparov's category $KK^G$ may be viewed as a triangulated category.
By using Brown's representability theorem for triangulated categories \cite{neeman1996}, a weakly isomorphic,
so-called Dirac
element $D \in KK^G(B,A)$
is constructed such that $B$ is a $G$-algebra in the localizing subcategory 
of $KK^G$ generated
by $G$-algebras of the form $\mbox{Ind}_H^G(F)$ (induction in the sense of Green \cite{green1978}) for a compact subgroup $H \subseteq G$ and $H$-algebra $F$.
If $G$ is compact then one will choose $B= \mbox{Ind}_G^G(A) = A$ and $D=id$, and for non-compact $G$
one hopes that the compactly induced algebras approximate $A$ sufficiently enough via $D$,
like one approximates 
functions vanishing at infinity
by compactly supported functions.
%
The Baum--Connes map turns out to be equivalent to the map $K(B \rtimes_r G) \rightarrow K(A \rtimes_r G)$
induced by $j^G_r(D) \in KK(B \rtimes_r G, A \rtimes_r G)$ for the descent homomorphism $j^G_r$.
Clearly, if for example the morphism
$D$ was an isomorphism then the functor image $j_r(D)$
would be an isomorphism as well and the Baum--Connes map bijective.

Let us observe the usefulness 
of this approach. Assume
for the moment that $B$ takes the particular simple form
$B= \mbox{Ind}_H^G(F)$. Then the left hand side of the new formulated Baum--Connes map is potentially computable
via
\begin{equation}    \label{equBCgreenjulg}
K(B \rtimes_r G) = K\big(\mbox{Ind}_H^G(F) \rtimes_r G \big) \cong K(F \rtimes_r H) \cong KK^H(\C,F)
\end{equation}
by Green's imprimitivity theorem \cite{green1978} and the Green--Julg isomorphism \cite{julg1981}.
Arbitrary $B$ might then be treated by homological means in triangulated categories.

In this paper we try to adapt the above method to unital, countable inverse semigroups $G$.
The compact subgroups are then the finite subinverse semigroups $H \subseteq G$.
In a former paper, \cite{burgiGreen}, we proved a Green imprimitivity theorem $\mbox{Ind}_H^G(F) \widehat{\rtimes} G \cong
F \widehat{\rtimes} H$ (Sieben's crossed product \cite{sieben1997}) for such $H$s.
Together with the Green--Julg isomorphism for inverse semigroups we get 
an analog identity to
(\ref{equBCgreenjulg}). 
The next fundamental step is to show that Kasparov's category $KK^G$ is a triangulated category.
Most of this goes literally through as in Meyer and Nest's paper \cite{meyernest}, and we collect the
definitions and facts in Section \ref{sectionTriangulated}. 
However, there is one exception.
To achieve that every morphism of $KK^G$ fits into an exact triangle, one needs a Cuntz-picture
of $KK^G$ by representing morphisms as $*$-homomorphisms. This was done in group equivariant
$KK$-theory by Meyer \cite{meyer}, and we adapt his proof in Section \ref{sectionCuntzpicture}. 
One problem
is that we need a model of a compatible $\ell^2(G)$-space,
and to construct it we need to impose a transparent (see Lemma \ref{lemmaEcontinuity}), but properly restricting condition
on $G$ which we call $E$-continuity.
%

The next step is to define an induction functor $\mbox{Ind}_H^G:KK^H \rightarrow KK^G$ for finite
subinverse semigroups $H \subseteq G$.
We do this in Section \ref{sectionInduction}. 
In Section \ref{KKtheory} we recall 
the definitions of $KK^G$-theory and fix other notions we shall need.
In Section \ref{sectionBott} we discuss Bott periodicity for $KK^G$.
In Section 
\ref{sectionDirac}
we believed that we had defined a Dirac element $D \in KK^G(P,\C)$
by an adaption of the corresponding proof in \cite{meyernest}.
Unfortunately, however, when finishing this paper closely in this form, we have realized that we had a flaw
in the proof of the fundamental identity
\begin{eqnarray}    \label{indresadjointidentity}
{K K^G}(\mbox{Ind}_H^G \, A, B)  &\cong&  K K^H(A, \mbox{Res}_G^H \, B),
\end{eqnarray}
which holds for discrete groups, see line (20) in \cite{meyernest}.
It is even wrong, 
see Remark \ref{remarkAdjointfunctorIndRes}.
%
On a sufficiently big subcategory
there exists a right adjoint functor to $\mbox{Ind}_H^G$ by theoretical results of Neeman, see Definition \ref{defRightadjointtoInd}, but it is not the restriction functor. 
We have no concrete realization of it
and consequently we cannot analyse it like the restriction functor.

Nevertheless, in the last Section \ref{sectionDirac}
we shall work with
the correct
right adjoint functor instead of the restriction functor and prove the existence
of a Dirac morphism under yet unverified assumptions on it  
which are  
essentially proved
for the
restriction functor in Section 
\ref{sectionMotivationDiracmorphism}.
%
%
We remark that the existence of the Dirac morphism is the main obstacle. After having it, one could easily construct a Baum--Connes map as in \cite{meyernest}, see Remark \ref{remarkBC}.



In the meanwhile, we have computed a right adjoint functor for the induction functor for a special
subclass of $G$-algebras called fibered $G$-algebras in \cite{burgiNoteBaumConnes}
and established a Baum--Connes map for them.
This is however 
not the
solution to the Baum--Connes map, as fibered $G$-algebras
do not appear to be $KK^G$-equivalent to such important examples of $G$-algebras like $C_0(X)$.
The fibered restriction functor developed there however has properties also outside of the class of fibered $G$-algebras which support the assumptions
imposed on the correct right adjoint functor, 
see Remark \ref{remarkBC}.

On the way of our attempt of proving the existence of a Dirac morphism we also showed
a number of lemmas in connection with restriction and induction functors which might be of independent
interest and are collected in Section \ref{sectionLemmasResInd}.

Some arguments are not completely verified; we always point out this and
write `Conjecture' after Lemma, Theorem etc.
(exactly these are Lemma \ref{lemmaFunctorInduction}, Theorem \ref{adaptionth65}
and Definition \ref{defRightadjointtoInd}).

\section{$G$-equivariant $KK$-theory}

\label{KKtheory}

Let $G$ denote a countable unital inverse semigroup.
We write $E(G)$ (or simply $E$) for the set of projections of $G$.
We shall denote
the involution
on $G$ both by $g \mapsto g^{*}$ and $g \mapsto g^{-1}$ (determined by $g g^{-1} g = g$).
A semigroup homomorphism is said to be {\em unital} if it preserves the identity $1 \in G$
and the zero element $0 \in G$ provided $G$ has such elements, respectively.
We consider $G$-equivariant $KK$-theory as defined in \cite{burgiSemimultiKK} (in its final form
in Section 7 of \cite{burgiSemimultiKK}) but make a slight adaption by making this theory
{\em compatible} in the following sense. 
We require that all $G$-Hilbert $A,B$-bimodules $\cale$ of Kasparov cycles satisfy
$e(a) \xi = a e(\xi)$ and $\xi e(b) = e(\xi) b$ for all $e \in E, a \in A, b \in B$ and $\xi \in \cale$.
Since the only constructions of Hilbert modules in \cite{burgiSemimultiKK} out of given
ones are done by forming tensor products,
direct sums, or taking the Hilbert module $\C$, and these constructions respect these modifications,
we readily can accept this modified, compatible $KK^G$-theory to hold true
with all its properties like the existence of the Kasparov product as in \cite{burgiSemimultiKK}.
Since the additional properties of inverse semigroups as compared to semimultiplicative sets in \cite{burgiSemimultiKK} slightly simplify the formal definitions of equivariant $KK$-theory (see for instance \cite[Corollary 4.6]{burgiDescent}),
we are going to recall the polished definitions for convenience of the reader.

\begin{definition}   \label{defCstar}
{\rm
A {\em $G$-algebra} $(A,\alpha)$ is a $\Z/2$-graded $C^*$-algebra $A$ with a
unital semigroup homomorphism
$\alpha: G \rightarrow \mbox{End}(A)$ such that
$\alpha_g$ respects the grading
and $\alpha_{g g^{-1}}(x) y = x \alpha_{g g^{-1}}(y)$
for all $x,y \in A$ and $g \in G$.
}
\end{definition}


\begin{definition}   \label{defHilbert}
{\rm
A {\em $G$-Hilbert $B$-module} $\cale$ is a $\Z/2$-graded Hilbert
module over a $G$-algebra $(B,\beta)$ endowed with a unital
semigroup homomorphism $G \rightarrow \mbox{Lin}(\cale)$ (linear maps on $\cale$)
such that $U_g$ respects the grading and
$\langle U_g(\xi),U_g(\eta)\rangle = \beta_g(\langle \xi,\eta \rangle)$,
$U_g(\xi b) = U_g(\xi) \beta_g(b)$,
and
$U_{g g^{-1}}(\xi) b = \xi \beta_{g g^{-1}}( b)$
for all $g \in G,\xi,\eta \in \cale$ and $b \in B$.
}
\end{definition}

In the last definition, $U_{g g^{-1}}$ is automatically a self-adjoint projection in the center of
$\call(\cale)$,
and the action $G \rightarrow \mbox{End}(\call(\cale))$, $g(T) =
U_g T U_{g^{-1}}$ turns $\call(\cale)$ to a $G$-algebra ($g \in G$ and $T \in \call(\cale)$).
A $G$-algebra $(A,\alpha)$
is a $G$-Hilbert module over itself under the inner product $\langle a,b\rangle
= a^* b$ and
$U: = \beta := \alpha$ in the last definition.
%
%
%
%
A $*$-homomorphism between $G$-algebras is called {\em $G$-equivariant} if it intertwines
the $G$-action.
Usually the $G$-action on a $G$-algebra is denoted by $g(a):= \alpha_g(a)$.
The complex numbers $\C$ are endowed with the trivial $G$-action $g(1)=1$ for all $g \in G$.
%
%
A {\em $G$-Hilbert $A,B$-bimodule} over $G$-algebras $A$ and $B$
is a $G$-Hilbert $B$-module $\cale$ equipped with a $G$-equivariant $*$-homomorphism
$A \rightarrow \call(\cale)$.




\begin{definition}  \label{defCycle}
{\rm
Let $A$ and $B$ be $G$-algebras.
We define a Kasparov cycle $(\cale,T)$, where $\cale$ is a $G$-Hilbert $A,B$-bimodule, to be an ordinary Kasparov cycle (without $G$-action) (see \cite{kasparov1981,kasparov1988}) satisfying $U_g T U_g^* - T U_{g g^{-1}} \in \{S \in \call(\cale)|\, a S, S a \in \calk(\cale) \mbox{ for all } a \in A\}$ for all
$g \in G$. The 
Kasparov group $KK^G(A,B)$ is defined to be the collection $\E^G(A,B)$ of these cycles
divided by homotopy induced by $\E^G(A,B[0,1])$.
}
\end{definition}

We write $C^*_G$ for the category of 
$G$-algebras as objects and $G$-equivariant $*$-homomorphisms as morphisms,
and $KK^G$ for the additive category consisting of 
$G$-algebras as objects and $KK^G(A,B)$ as the morphism set
from object $A$ to object $B$, together with the Kasparov product $KK^G(A,B) \times KK^G(B,C)
\rightarrow KK^G(A,C)$ as composition of morphisms.
Define $C_G:C^*_G \rightarrow KK^G$ to be the well known functor 
which is identical on objects and satisfies
$C_G(f):= f_*(1_A) \in KK^G(A,B)$ for morphisms $f:A \rightarrow B$, where $1_A := [(A,0)] \in KK^G(A,A)$
denotes the unit.

\begin{definition}[See Definition 25 of \cite{burgiSemimultiKK}]    \label{defTau}
{\rm
For a $\sigma$-unital $G$-algebra $D$ we denote by $\tau_D:KK^G(A,B) \rightarrow KK^G(A \otimes D,B \otimes D)$
the map induced by $(\cale,T) \mapsto (\cale \otimes D,T \otimes 1)$.
}
\end{definition}

Occasionally we shall still refer to {\em incompatible} $KK^G$-theory as defined in \cite{burgiSemimultiKK}
and denote it by $IK^G$. The class of underlying $G$-Hilbert modules is richer, but the $G$-algebras are the same.
%
%
%
%
$KK^G$ and their Hilbert modules are sometimes accompanied by the word {\em compatible},
to stress the difference to $IK^G$.
It is often useful
to compare $IK^G$ and $KK^G$ by the isomorphism
$IK^G(A,B) \cong KK^G(A \rtimes E, B \rtimes E)$ from \cite[Theorem 5.3]{burgiGreenjulg}
for {\em finite} $G$.
Also remark that there exists a canonical functor $KK^G \rightarrow IK^G$
defined by the identity map on cycles.

Given a $G$-algebra $A$,
we denote by $A \rtimes G$ the universal crossed product \cite{khoshkamskandalis2004}, 
and by $A \widehat \rtimes G$ Sieben's crossed product \cite{sieben1997}.
We identify $G$ as a subset of $\C \rtimes G$, and denote by $\tilde G \subseteq \C \rtimes G$
the inverse semigroup generated by $G$ and all projections $p \in \C \rtimes G$ of the form
$p=e_0 (1 -e_1) \ldots (1-e_n)$ for $e_i \in E$ and $n \ge 0$.
Note that every element of $\tilde G$ is of the form $g p$ with $g \in G$ and $p$ as before.

Every $G$-action $\alpha$ on a $G$-algebra (or $G$-Hilbert module) extends to a $\tilde G$-action
by linearity, that is, $\alpha_{g p} = \alpha_g \alpha_{e_0} (\alpha_1-\alpha_{e_1}) \ldots (\alpha_1-\alpha_{e_n})$, where
$p$ is as before (see \cite[Lemma 2.1]{burgiGreen}). We sometimes extend $G$-actions to $\tilde G$-actions in this way implicitly without saying.
%
%
%
%
%
We shall also consider discrete groupoids $H \subseteq \tilde G$, and we may regard them as inverse semigroups
$H \cup \{0\} \subseteq \tilde G$ with zero element 
in order to consistently redefine the known notion of
$H$-equivariant $KK$-theory $KK^H$ via the inverse semigroup $H \cup \{0\}$, 
where $0$ is understood to act always as zero.
%
Provided 
is here
however that the $H$-algebras are defined in the groupoid sense,
that is, that they are also $C_0(H^{(0)})$-algebras, 
see \cite[Definition 1.5]{kasparov1988}. (Cf. also \cite{burgiKKrDiscrete}.)
%
%

Let $G \subseteq L \subseteq \tilde G$ be a subinverse semigroup.
Then we have
\begin{equation}   \label{equKKtildeG}
KK^G(A,B) = KK^{L}(A,B) = KK^{\tilde G}(A,B)
\end{equation}
via the identity map on cycles when using the 
above mentioned extension of $G$-actions for all $G$-algebras $A$ and $B$.
(A $\tilde G$-Hilbert $B$-module inherits the linearly extended $\tilde G$-action from $B$
by compatibility.)
%
%
Denote by $X$ or $X_G$ the totally disconnected, locally compact Hausdorff space such that $C_0(X)$ 
is the universal commutative $C^*$-algebra $C^*(E)$ generated by the commuting projections $E$.
(Actually $X$ is compact since $E$ is unital.) $C_0(X)$ is endowed with the $G$-action $g(1_e)= 1_{g e g^*}$ for
$e \in E$ and $g \in G$.
Every $G$-algebra $A$ may be regarded as a $C_0(X)$-algebra (see Kasparov \cite[Section 1.5]{kasparov1988}) by $\pi:C_0(X) \rightarrow Z(\calm(A))$
with $\pi(1_e)(a)=e(a)$ since $E$ has a unit. 
Write $A \otimes^X B$ 
for the balanced tensor product ($A \otimes B$ divided by all elements of the form $e(a) \otimes b - a \otimes e(b)$
where $e \in E$),
see Le Gall \cite{legall1999} or \cite[Section 1.6]{kasparov1988}.


%
%

\begin{definition}
{\rm
The {\em groupoid $H \subseteq \tilde G$ associated to a given finite subinverse semigroup $H' \subseteq G$} is defined to be
the finite groupoid
$H=\{h p \in \tilde G\,|\, h \in H', p \in E(\widetilde{H'}) \mbox{ is a minimal projection}, \, h^*h \ge p\}$.
}
\end{definition}
Observe that $KK^{H'}(A,B) = KK^H(A,B)$ for all $H'$-algebras or $H$-algebras $A$ and $B$
by the equivalence of $C^*_{H'}$ and $C^*_{H}$, and $KK^{H'}$ and $KK^H$, respectively, see
\cite{burgiKKrDiscrete}. (Our notion $KK^{H'}$ coincides with $\widehat{KK^{H'}}$ of \cite{burgiKKrDiscrete}.)
All subinverse semigroups of $G$ are assumed to contain the {\em unit} of $G$!
By regarding $G$ as a discrete inverse semigroup,
we often say compact instead of finite subinverse semigroup.

\section{Bott periodicity}

\label{sectionBott}

This section works both in $IK^G$ and $KK^G$.

\begin{definition}
{\rm
Define $KK^G_n (A,B) := KK^G(A \otimes C_{n,0},B)$,
where $C_{n,m}$ denotes the Clifford algebras of Kasparov \cite[Sections 2.11 and 2.13]{kasparov1981}
for $n,m \ge 0$.
(The $G$-action on $C_{n,0}$ is trivial.)
}
\end{definition}


\begin{theorem}[Bott periodicity]     \label{theoremBottperiodicity}

Let the $G$-action on $C_0(\R^n)$ be trivial.
Then
$$KK^G_{i+n}(A \otimes C_0(\R^n),B)  \cong   KK^G_{i}(A,B)
 \cong  KK^G_{i-n}(A,B \otimes C_0(\R^n))$$

\end{theorem}

\begin{proof}
%

The proof is a slight adaption of Kasparov's \cite[\S 5, Theorem 7]{kasparov1981}.
Note that Kasparov discusses in his proof the ``real" case to be definite, and so
our $\R^n$ appears as $\R^{p,q}$ in his proof; so we ``identify" these two.
In line (4) on page 547 of \cite{kasparov1981} he states that there exists
elements $\beta_V \in KK^{Spin(V)}(\C,C_0(\R^n) \otimes C_V)$
and $\alpha_V \in KK^{Spin(V)}(C_0(\R^n) \otimes C_V,\C)$
such that
\begin{equation} \label{tintan}
a)\; \beta_V \otimes_{C_0(\R^n) \otimes C_V} \alpha_V = c_1; \qquad
b)\; \beta_V \otimes_{\C} \alpha_V = \tau_{C_0(\R^n) \otimes C_V}(c_1),
\end{equation}
where $c_1 := (id, \C,0) \in KK^{Spin(V)}(\C,\C)$ is the unit element, and the Kasparov products
in (\ref{tintan}) are the Kasparov's cup-cap product.
As Kasparov remarks, a direct application of (\ref{tintan}) to \cite[\S 4, Theorem 6, 2)]{kasparov1981}
yields the desired Bott periodicity \cite[\S 5, Theorem 5]{kasparov1981}.

We now regard $\beta_V$ and $\alpha_V$ as elements in $G$-equivariant $KK$-theory $KK^G$
by putting them into the canonical map $KK^{Spin(V)}(C,D) \rightarrow KK^G(C,D)$ ($\forall C,D$) by regarding
$Spin(V)$-Kasparov cycles as $G$-Kasparov cycles via the trivial semigroup homomorphism $triv:G \rightarrow Spin(V): g \mapsto 1$
($\forall g \in G$).
We can then also apply (\ref{tintan}) to \cite[\S 4, Theorem 6, 2)]{kasparov1981}, but now in the $G$-equivariant
setting.
\end{proof}

\begin{corollary}
We have
$KK^G(A \otimes C(\R^2),B) \cong KK^G(A,B) \cong KK^G(A,B \otimes C(\R^2))$ for all $G$-algebras $A$ and $B$.
\end{corollary}

\begin{proof}
The Clifford algebra $C_{0,0}$ is $\C$, so that $KK_0^G$ is simply $KK^G$.
The result follows then from Theorem \ref{theoremBottperiodicity} and the formal
Bott periodicity \cite[Theorem 5.5]{kasparov1981} (which works literally in our setting as
the $G$-actions on the vector spaces $V$ appearing there are trivial), which states that $KK_n$ is periodic in $n$
with period $2$.
%
\end{proof}

\section{Induction and restriction functors}

\label{sectionInduction}


Given a compact subinverse semigroup $H' \subseteq G$,
in \cite{burgiGreen} we defined an induced algebra 
and showed
Green imprimitivity theorems. This was done by switching at first from $H'$ to its
associated finite subgroupoid $H \subseteq \tilde G$, proving everything for $H$,
and at the end switching back to $H'$ in notation. That $H$ was induced by an inverse semigroup
was extraneous. Hence we may, and shall, start here somewhat more generally with a finite groupoid
like in Definition \ref{defInductionAlg} below and still can use the results from \cite{burgiGreen}.

Before we need however fix some notions.
For an assertion $\cala$ we let $[\cala]$ be the real number $0$ if $\cala$ is false, and $1$ if $\cala$ is true.
Let $H \subseteq \tilde G$ be a finite subgroupoid.
Set
$$G_H:= \{g p \in \tilde G\,|\, g \in G,\, p \in H^{(0)},\, g^* g \ge p\}.$$
We endow $G_H$ with an equivalence relation:
$g \equiv h$ if and only if there exists $t \in H$ such that $g t = h$ ($g,h \in G_H$).
We denote by $G_H/H$ the discrete, set-theoretical quotient of $G_H$ by $\equiv$. 
%
The delta function $\delta_g$ in $C_0(G_H)$ and $C_0(G_H/H)$ is denoted by $g$ ($g \in G_H$).
The commutative $C^*$-algebras $C_0(G_H)$ and $C_0(G_H/H)$ are endowed with the $G$-action $g(h):= [g h \in G_H] \,\, g h$, where $g \in G$ and $h \in G_H$
(of course, $g h \in G_H$ is equivalent to $g^* g \ge h h^*$).
%


\begin{definition}   \label{defInductionAlg}
{\rm
Let $H \subseteq \tilde G$ be a finite subgroupoid
and $D$ a $H$-algebra.
Define, similar as in \cite[\S 5 Def. 2]{kasparov1995},
\begin{eqnarray*}
\mbox{Ind}_{H}^G(D)  &:=&  \{ f: G_H \rightarrow D \, | \,\, \forall g \in G_H, t \in H \mbox{ with } gt \in G_H: f(g t) = t^{-1}(f(g)),\\ 
&& \qquad \|f(g)\| \rightarrow 0 \mbox{ for } g H \rightarrow \infty \mbox{ in } G_H/H \, \}.
\end{eqnarray*}
It is a $C^*$-algebra under the pointwise operations and the supremum's norm and becomes a $G$-algebra
under the $G$-action
$(g f)(h) := [g^{-1} h \in G_H ] \,\,f ( g^{-1} h)$
for $g \in G$, $h \in G_H$ and $f \in \mbox{Ind}_H^G(D)$.
}
\end{definition}

\begin{definition}    \label{defInductionHom}
{\rm
Let $H \subseteq \tilde G$ be a finite subgroupoid. Define a functor ${\cal I}_H^G: C^*_H \rightarrow C^*_G$
by ${\cal I}_H^G(A) = \mbox{Ind}_H^G(A)$ for objects $A$ in $C^*_H$ and
${\cal I}_H^G(f): \mbox{Ind}_H^G(A) \rightarrow \mbox{Ind}_H^G(B)$ by ${\cal I}_H^G(f)(x) = f(x(g))$ for morphisms $f:A \rightarrow B$ in $C^*_H$, where $x \in \mbox{Ind}_H^G(A)$
and $g \in G_H$.
}
\end{definition}

%

\begin{lemma}  \label{lemmaInductionfunctorIntertwines}
The functor ${\cal I}_H^G$ is exact, and canonically intertwines 
direct sums (i.e. $\mbox{Ind}_H^G(\bigoplus_i {A_i}) \cong \bigoplus_i \mbox{Ind}_H^G({A_i})$), tensoring with a nuclear $C^*$-algebra
$B$ endowed with the trivial $G$-action (i.e. ${\cal I}_H^G(A \otimes B) \cong {\cal I}_H^G(A) \otimes B$) and
the mapping cone (see (\ref{defmappingcone})) (i.e. $\mbox{Ind}_H^G(\mbox{cone}(f)) \cong \mbox{cone}(\mbox{Ind}_H^G(f))$).
%
\end{lemma}

\begin{proof}
The proof is straightforward.
\end{proof}

%

Define $C_0(G_H/H,B)$ to be the $G$-invariant ideal of $C_0(G_H/H) \otimes B$ which is the closure of the linear span of
all elements of the form $g \otimes g g^*(b)$ ($g \in G_H, b \in B$).
Similarly, denote by $p \in Z (\call(\mbox{Ind}_H^G(A) \otimes B))$ (center) the central projection $p(g \otimes a \otimes b):= g \otimes a \otimes
g g^*(b)$ for $g \in G_H, a \in g^*g(A)$ and $b \in B$.
We have a direct sum decomposition
\begin{eqnarray}   \label{decomposIndHGtensor}
\mbox{Ind}_H^G(A) \otimes B &\cong& p \big(\mbox{Ind}_H^G(A) \otimes B \big)
\oplus (1-p)\big(\mbox{Ind}_H^G(A) \otimes B \big),
\end{eqnarray}
and we denote the first summand (and ideal) by $\mbox{Ind}_H^G(A) \stackrel{\rightarrow}{\otimes} B$.
%
%

\begin{lemma}[Cf. line (17) in \cite{meyernest}]    \label{lemmaResIndandC}
Let $B$ be a $G$-algebra and $H \subseteq \tilde G$ a finite subgroupoid.
Then there is a $G$-equivariant $*$-isomorphism 
$$\Theta:\mbox{Ind}_H^G \mbox{Res}_G^H (B)  \longrightarrow
C_0(G_H/H, B), \quad
\Theta(f)  \,\, = \,\, \sum_{ g \in G_H/H}  g \otimes g(f(g))$$
for all $f \in \mbox{Ind}_H^G \mbox{Res}_G^H (B) \subseteq C_0(G_H) \otimes B$.
(The sum is understood that we choose for every equivalence class in $G_H/H$
exactly one arbitrary representative $g \in G_H$.)
\end{lemma}

\begin{proof}
The proof is straightforward.
\end{proof}

%

\begin{lemma}[Cf. line (16) in \cite{meyernest}]    \label{lemmaResIndandC2}
Let $H \subseteq \tilde G$ be a finite subgroupoid, $A$ a $H$-algebra
and $B$ a $G$-algebra.
Then there is a $G$-equivariant $*$-isomorphism 
$$\Theta:\mbox{Ind}_H^G \big(A \otimes^{X_H} \mbox{Res}_G^H (B) \big)  \longrightarrow
\mbox{Ind}_H^G(A) \stackrel{\rightarrow}{\otimes} B,
\quad
\Theta( g \otimes a \otimes b ) =  g \otimes a \otimes g(b)$$
for all $g \in G_H, a \in g^*g(A)$ and $b \in g^*g(B)$.
\end{lemma}

\begin{proof}
The tensor product $A \otimes^{X_H} \mbox{Res}_G^H (B)$ denotes the balanced groupoid tensor product
and is endowed with
the diagonal $H$-action.
In other words, we may regard $A$ and $\mbox{Res}_G^H (B)$ as
$H \cup \{0\}$-inverse semigroup algebras and take the usual diagonal inverse semigroup action
for the tensor product $A \otimes^{X_{H \cup \{0\}}} \mbox{Res}_{\tilde G}^{H \cup \{0\}} (B)$.
%

Note that we have $g t\otimes t^*(a \otimes b) = gt \otimes t^*(a) \otimes t^*(b)$
in $\mbox{Ind}_H^G \big(A \otimes^{X_H} \mbox{Res}_G^H (B) \big)
\subseteq C_0(G_H) \otimes A \otimes B$ for all $g \in G_H, t \in H, a \in A$ and $b \in B$
with $gt \in G_H$, so we can achieve the required format in the argument of $\Theta$ when setting $t:= g^*g$.
Surjectivity of $\Theta$ is obvious.
That $\Theta$ is isometric is also clear as the transition $g^*g B \rightarrow gB$ by $\Theta$ is a $*$-isomorphism.
\end{proof}

%
%


From now on we restrict ourselves to trivially graded $G$-algebras.
The next lemma must be viewed as conjecture as we were not be able to completely verify the
stability property.

\begin{lemma}[Conjecture]   \label{lemmaFunctorInduction}
The functor $F= {C_G} \circ {\cal I}_H^G$ from the category $C^*_H$ to the additive category $KK^G$
%
is a stable, split exact and homotopy invariant functor.
(Stability means that $F(f:A \rightarrow A \otimes \calk)$ is an isomorphism for every corner embedding $f$,
where $A \otimes \calk$ is allowed to be equipped with any $H$-action.)
\end{lemma}

\begin{proof} 
By Higson \cite[Section 4.4]{higson},
we need to show that the functor $L: {C^*_H} \rightarrow {Ab}$ determined
by $L(B) = KK^G(A, {\cal I}_H^G(B))$ for objects $B$ and $L(f)= \cali_H^G(f)_*:KK^H(A,
{\cal I}_H^G(B_1)) \rightarrow KK^H(A,{\cal I}_H^G(B_2))$ for morphisms $f:B_1 \rightarrow B_2$
is a stable, split exact and homotopy invariant functor for all objects $A$ in $KK^G$ in the sense of \cite{burgiUniversalKK}.
This follows from Lemma \ref{lemmaInductionfunctorIntertwines}
and \cite[Proposition 1.1]{burgiUniversalKK}, which says that
the functor $B \mapsto KK^G(A,B)$ is stable, split exact and homotopy invariant.
Note that $\calk$, being simple, allows only $G$-actions by automorphisms
(since $g g^* \calk$ is an ideal in $\calk$).
Some gap is here that we require the $G$-action on $B \otimes \calk$ to be diagonal
and by Lemma \ref{lemmaInductionfunctorIntertwines} we can allow the action on $\calk$ only be trivial,
however, in \cite{burgiUniversalKK} it could be anyone. 
%
%
\end{proof}

Because $F$ is 
stable, split exact and homotopy invariant, it factors through $KK^H$ by \cite[Theorem 1.3]{burgiUniversalKK}
and this gives us a new functor defined next.
We remark that \cite[Theorem 1.3]{burgiUniversalKK} works also for countable discrete groupoids $H$, as pointed out in \cite{burgiUniversalKK},
by regarding $H \cup \{0\}$ as an inverse semigroup with zero element.

\begin{definition}   \label{defInduction}
{\rm
Let $H \subseteq \tilde G$ be a finite subgroupoid. We define the {\em induction functor}
$\mbox{Ind}_H^G : KK^H \rightarrow KK^G$ as the unique functor
satisfying
${C_G} \circ {\cal I}_H^G = \mbox{Ind}_H^G \circ { C_H}$, see \cite[Theorem 1.3]{burgiUniversalKK}
and Lemma \ref{lemmaFunctorInduction}.
}
\end{definition}

%
%
If $H' \subseteq G$ is a finite subinverse semigroup then we consider its associated finite subgroupoid
$H \subseteq \tilde G$ and define induction by $\mbox{Ind}_{H'}^G:=\mbox{Ind}_{H}^G$;
usually we regard it, however, as a functor $\mbox{Ind}_{H'}^G: KK^{H'} \rightarrow KK^G$.

Even we have not checked this, it seems likely that the induction functor is given on the level of cycles
analogously 
as the induction functor for groups \cite{kasparov}.
That is, we conjecture that
$$\mbox{Ind}_H^G(\cale,T) = (\mbox{Ind}_H^G(\cale), \mbox{Ind}_H^G(T)),$$ 
where for the $G$-Hilbert $A,B$-bimodule $\cale$ a dense subspace of $\mbox{Ind}_H^G(\cale)$ is verbatim defined like in Definition \ref{defInductionAlg}
(consiting only of functions on $G_H$ with finite carrier),
with pointwise (over the points $G_H$) module operations and inner product and $G$-action 
verbatim as in Definition \ref{defInductionAlg}. 
Its closure should turn 
it into a $G$-Hilbert $\mbox{Ind}_H^G(A),\mbox{Ind}_H^G(B)$-bimodule $\mbox{Ind}_H^G(\cale)$,
and
the operator $\mbox{Ind}_H^G(T)$ should be given by  
$$\mbox{Ind}_H^G(T) \xi (g) = \sum_{h \in H, \, g h \in G_H} h (T(\xi(gh))) \quad (\xi \in \mbox{Ind}_H^G(\cale),\,g \in G_H).$$
We shall not use this in this paper.


\begin{definition}
{\rm
Let $H \subseteq G$ be a subinverse semigroup or $H \subseteq \tilde G$ a finite subgroupoid.
The {\em restriction functor} $\mbox{Res}_G^H: KK^G \rightarrow KK^H$
is defined by restricting $G$-actions (or $\tilde G$-action for the groupoid
$H$) to $H$-actions in $G$-algebras and $G$-Hilbert modules of cycles.
%
Additionally, every restricted $H$-algebra is cut-down to the form $\mbox{Res}_G^H(A) = 1_H(A)$ in case that $H$ is a groupoid
($1_H := \sum_{x \in H^{(0)}} x$) or $H$ should not contain the identity of $G$.
}
\end{definition}

\begin{remark}    \label{remarkAdjointfunctorIndRes}
{\em
Identity (\ref{indresadjointidentity}) is wrong in $KK^G$. Take for example a finite, unital inverse
semigroup $G$ where no other projection than $1$ is connected with $1$. Set $H=\{1\}$, and $A=B=\C$
endowed with the trivial $G$-action.
Then $KK^G(\mbox{Ind}_H^G\, \C,\C)=0$, because a cycle $(\cale,T)$ satisfies
$a \xi 1(b) = a \xi p(b) = p(a) \xi b = 0$ for all $a \in \mbox{Ind}_H^G(\C), b \in \C, \xi \in \cale$ and
any projection
$p<1$ in $E$. But $KK^H(\C,\mbox{Res}_G^H \,\C) = \Z$.


Identity (\ref{indresadjointidentity}) is also wrong in $IK^G$.
Let $G=E$ be finite and consist only of projections. Set $H=\{e\}$, where $e$ denotes the minimal
projection of $E$. Then $\mbox{Ind}_H^E\, \C \cong \C$ and thus
$IK^E(\mbox{Ind}_H^E\, \C,\C) \cong K(\C \rtimes E) \cong \Z^m$
by the Green--Julg isomorphism in \cite{burgiGreenjulg}.
But $IK^H(\C , \mbox{Res}_G^H\, \C) \cong KK(\C,\C) \cong \Z$.

}
\end{remark}

\section{Realizing morphisms in $KK^G$ by $*$-homomorphisms}

\label{sectionCuntzpicture}

Generalizing the Cuntz picture of $KK$-theory, \cite{cuntz1987}, to equivariant $KK$-theory,
Meyer showed in \cite[Theorem 6.5]{meyer} that for every locally compact second countable group $G$
and for every morphism $x \in KK^G(A,B)$ there exist $G$-algebras $A'$ and $B'$,
isomorphisms $y \in KK^G(A,A')$ and $z \in KK^G(B,B')$, and a $*$-homomorphism
$f: A' \rightarrow B'$ (also interpreted as an morphism in $KK^G$)
such that $x = z \circ f \circ y^{-1}$. That is, we may rewrite morphisms in $KK^G$ as $*$-homomorphisms.
We will adapt Meyer's proof to the case of an inverse semigroup $G$ (see Theorem \ref{adaptionth65}).
To this end, we need a model for an $\ell^2(G)$-space, since it plays a central role in Meyer's work \cite{meyer}.
However, a direct translation from a group $G$ to an inverse semigroup $G$ does not work, even not if taking
the $\ell^2(G)$ from Khoshkam and Skandalis \cite{khoshkamskandalis2004}, since it is a useful
{\em incompatible} $\C$-module, however, we need a {\em compatible} model
for $\ell^2(G)$, that is, a compatible $G$-Hilbert $C_0(X)$-modul.
%
This is necessary as to achieve that the action $g g^{-1}$ ($g \in G$) is in the center of $\call(\cale)$
in all derived spaces $\cale$ from $\ell^2(G)$ 
and consequently the $G$-action on
$\call(\cale)$ is 
multiplicative and so a $G$-action.
Hence constructions like $q_s A:= q(\K(G \N) A)$ in \cite{meyer}
or Definitions \ref{defintionKGA} and \ref{defintionKGNA}
become indeed $G$-algebras as required.

%
%
%

In the next few paragraphs (until Definition \ref{defL2hilbertmodule}) we shall identify elements $e \in E$ with its characteristic function $1_e$ in $C_0(X)$.
Write $\mbox{Alg}^*(E)$ for the dense $*$-subalgebra of $C_0(X)$ generated by the characteristic functions
$1_e$ for all $e \in E$.
Moreover,
write $\bigvee_i f_i \in \C^X$ for the pointwise supremum of a family of functions $f_i: X \rightarrow \C$.
We shall use the order relation on $G$ defined by $g \le h$ iff $g = eh$ for some $e \in E$.

\begin{definition}
{\rm
An inverse semigroup $G$ is called {\em $E$-continuous} if
the function $\bigvee \{e \in E |\, e \le g\} \in \C^X$ 
is a {\em continuous} function in $C_0(X)$
for all $g \in G$.
}
\end{definition}

\begin{lemma}   \label{lemmaEcontinuity}
An inverse semigroup $G$ is {$E$-continuous} if and only if for every $g \in G$ there exists a finite
subset $F \subseteq E$ such that
$\bigvee \{e \in E |\, e \le g\} =\bigvee \{e \in F |\, e \le g\}$.
\end{lemma}

\begin{proof}
If $\bigvee \{e \in E |\, e \le g\} = 1_K \in C_0(X)$ for a clopen subset $K \subseteq X$ then $K$ must be compact. Hence $K = \bigcup \{\,\mbox{carrier}(1_e) \subseteq X\,| \, e \in E ,\, e \le g\}$ allows a finite subcovering,
where carrier denotes the usual carrier of a function on a locally compact space.
\end{proof}

\begin{definition}[Compatible $L^2(G)$-space]     \label{defCompatibleL2}
{\rm
Let $G$ be an $E$-continuous inverse semigroup.
Write $c$ for the linear span of all functions $\varphi_g: G \rightarrow \C$ (in the linear space $\C^G$) defined by
\begin{eqnarray*}
\varphi_g(t) & :=& [t \le g] 
\end{eqnarray*}
for all $g,t \in G$.
Endow $c$ with the $G$-action $g(\varphi_h) := \varphi_{g h}$ for all $g,h \in G$.
Turn $c$ to an $\mbox{Alg}^*(E)$-module by setting $\xi e := e(\xi)$ for all $\xi \in c$ and $e \in E$.
Define an $\mbox{Alg}^*(E)$-valued inner product on $c$ by
\begin{eqnarray}  \label{identinnerprodphi}
\langle \varphi_g, \varphi_h \rangle &:=& \bigvee \{e \in E \,|\, eg = e h ,
\,e \le g g^{-1} h h^{-1} \}.
\end{eqnarray}
The norm completion of $c$ is a 
$G$-Hilbert $C_0(X)$-module denoted by $\widehat{\ell^2}(G)$.
}
\end{definition}

We discuss the last definition.
At first notice that
$\langle \varphi_g , \varphi_h\rangle = g g^{-1} \bigvee \{e \in E|\,e = e h g^{-1}\}$
(observe that $e = e h g^{-1}$ implies $e \le h g^{-1} g h^{-1}$),
so that by $E$-continuity $\langle \varphi_g , \varphi_h\rangle$ is in $C_0(X)$ and actually even in $\mbox{Alg}^*(E)$ by
Lemma \ref{lemmaEcontinuity}, and $e \in E$ in (\ref{identinnerprodphi}) can be replaced by $e \in F$
for some finite subset $F \subseteq E$.
The identities
$\langle \varphi_g , \varphi_h \rangle = \langle \varphi_h , \varphi_g \rangle$, $\langle \varphi_g , \varphi_h f \rangle
= \langle \varphi_g f , \varphi_h \rangle
= \langle \varphi_g , \varphi_h \rangle f$,
$j(\langle \varphi_g , \varphi_h \rangle) = \langle j(\varphi_g) , j(\varphi_h) \rangle$
for all $g,h,j \in G$ and $f \in E$ are easy to check.
We note that (\ref{identinnerprodphi}) is positive definite. Indeed, assume
$\langle x,x\rangle = 0$ for $x=\sum_{i=1}^n \lambda_i \varphi_{g_i}$
with nonzero $\lambda_i \in \C$ and $g_i \in G$ mutually different.
Choose $g_j$ such that no other $g_i$ satisfies $g_j g_j^{-1} < g_i g_i^{-1}$.
Hence, $\langle \varphi_{g_j},\varphi_{g_j}\rangle = g_j g_j^{-1}$ but 
$\langle \varphi_{g_i},\varphi_{g_k}\rangle \neq g_j g_j^{-1}$ for all combinations where $i \neq k$.
By linear independence of the projections $E$ in $\mbox{Alg}^*(E)$ $\lambda_j$ must be zero;
contradiction.
The last proof also shows the following lemma.

\begin{lemma}
The vectors $(\varphi_g)_{g \in G} \subseteq \widehat{\ell^2}(G)$ are linearly independent.
\end{lemma}



\begin{definition}  \label{defL2hilbertmodule}
{\rm
Let $\cale$ be a $G$-Hilbert $B$-module. Then
$\widehat{\ell^2}(G,\cale) := \widehat{\ell^2}(G) \otimes^X \cale$ is a $G$-Hilbert $B$-module, where 
$\otimes^X$ denotes the $C_0(X)$-balanced exterior tensor product as defined by Le Gall \cite[Definition 4.2]{legall1999}
(or in this case equivalently, the internal tensor product $\otimes_{C_0(X)}$).
}
\end{definition}

Everywhere in \cite{meyer} we have to replace
$L^2(G)$ (see \cite[Section 2]{meyer}) by $\widehat{\ell^2}(G)$ and $L^2(G,\cale)$ (see \cite[Section 2.1.1]{meyer}) by $\widehat{\ell^2}(G,\cale)$. 
These definitions have to go further.

\begin{definition}   \label{defHilbertCXmodules}
{\rm
Every separable $G$-Hilbert space $\calh$ in Meyer \cite{meyer}
has to be replaced by a countably generated $G$-Hilbert $C_0(X)$-module $\calh$.
Every occurrence of the Hilbert space $\C$ in \cite{meyer} has to be substituted by the $G$-Hilbert $C_0(X)$-module $C_0(X)$.
For every $G$-Hilbert $B$-module or $G$-algebra $\cale$, $\ell^2(\calh) \otimes \cale$ in \cite{meyer}
has to be replaced by the compatible tensor product
$\ell^2(\calh) \otimes^X \cale$, and likewise $\K(\calh) \otimes \cale$ in \cite{meyer} by
$\K(\calh) \otimes^X \cale$.
}
\end{definition}

In the beginning of Section 3 of \cite{meyer} we have the following adaption.

\begin{definition}  \label{defGequivariance}
{\rm
Let $A$ and $B$ be $\sigma$-unital $G_2$-$C^*$-algebras and let $\calh$ be
a countably generated $G_2$-Hilbert $C_0(X)$-module.
A Kasparov triple $(\cale,\phi,F)$ is called {\em $\calh$-special} iff
\begin{itemize}
\item[(i)] $F$ is a $G$-equivariant symmetry 
({\em $G$-equivariance} means that the function $F: \cale \rightarrow \cale$ commutes with the
$G$-action $U_g :\cale \rightarrow \cale$ for all $g \in G$), and

\item[(ii)] $\calh \otimes^X \cale \subseteq \hat \calh_B$.

\end{itemize}
}
\end{definition}


\begin{lemma}     \label{lemma31meyer}
Lemma 3.1 of \cite{meyer} holds true also for an inverse semigroup $G$.
\end{lemma}

\begin{proof}
Let $(\cale,\phi,F)$ be an essential Kasparov triple for $A,B$.
Rather than the definition $F':C_c(G,\cale) \rightarrow C_c(G,\cale)$ ($(F'f)(g) = g(F) (f(g))$, $g \in G,f \in C_c(G,\cale)$)
in Meyer \cite{meyer} we have to use the following one.
Define $F':\widehat{\ell^2}(G) \otimes^X \cale \rightarrow \widehat{\ell^2}(G) \otimes^X \cale$
by
\begin{eqnarray*}
F'(\varphi_g \otimes \xi) &:=& \varphi_g \otimes g(F)(\xi)
\end{eqnarray*}
for $g \in G,\xi \in \cale$. We show that $F'$ is $G$-equivariant (see Definition \ref{defGequivariance}).
For $h \in G$ we have
\begin{eqnarray*}
h\big(F'(\varphi_g \otimes \xi )\big ) &=& h \varphi_{g} \otimes h g F g^{-1} h^{-1} h (\xi) \\
&=& \varphi_{h g} \otimes hg(F) (h(\xi)) \\
&=& F'\big (h(\varphi_g \otimes \xi) \big ),
\end{eqnarray*}
because $h^{-1} h \in \call(\cale)$ is in the center.

We have to check that $F'$ is an $F$-connection (see \cite[Section 2.5]{meyer})
when writing $\widehat{\ell^2}(G,\cale) \cong \widehat{\ell^2}(G,A) \otimes_A \cale$ (because $\phi$ is essential).
Write $\tau$ for the grading automorphisms on $A$ and $\widehat{\ell^2}(G,A)$.
Let $\xi := \varphi_g \otimes a 
\in \widehat{\ell^2}(G,A)$ for $g \in G$ and $a \in A$ with $g g^{-1}(a) = a$ without loss of generality.
Set $K:= T_\xi F-F'T_{\xi \tau}: \cale \rightarrow \widehat{\ell^2}(G,\cale)$ (see \cite[Section 2.5]{meyer})
for $T_\xi(\eta) = \xi \otimes \eta$ and $\eta \in \cale$. 
Then we have
\begin{eqnarray*}
K \eta &=& \varphi_g \otimes \phi(a) F \eta - \varphi_g \otimes g(F) \phi \tau(a) \eta
\; = \; \varphi_g \otimes K_g (\eta)
\end{eqnarray*}
in the space $\widehat{\ell^2}(G) \otimes^X \cale$ for all $\eta \in \cale$, where
\begin{eqnarray*}
K_g &:=& \phi(a) g g^{-1}(F) - g(F) \phi \tau (a) \;=\; [\phi(a),F] + \big (g g^{-1}(F)- g(F) \big) \phi \tau(a),
\end{eqnarray*}
because $a = g g^{-1}(a)$ and $g g^{-1} \in \call(\cale)$ is in the center and so $\phi(a) F = \phi(a) g g^{-1}(F)$. Since $(\cale,\phi,F)$ is a Kasparov triple, $K_g \in \calk(\cale)$.
Assuming for the moment that $K_g$ was an elementary compact operator $\theta_{\alpha,\beta}$ for $\alpha,\beta \in \cale$, we would have
$K = \varphi_g \otimes \theta_{\alpha,\beta} = \theta_{\varphi_g \otimes \alpha,\beta} \in \calk(\cale,\widehat{\ell^2}(G,\cale))$
as required. This is also true for general $K_g$ by approximation.
\end{proof}

\begin{definition}   \label{defintionKGA}
Instead of $\K(G)A := \K(L^2(G)) \otimes A$ in Proposition 3.2 (and Section 2.1.1) of Meyer's paper \cite{meyer}
we have to use $\K(G)A := \K(\widehat{\ell^2}(G)) \otimes^X A$.
\end{definition}

Note $\K(G)A$ is a $G$-algebra.
We have also an isomorphism of $G$-algebras
\begin{equation}   \label{isol2a}
\psi: \K(G)A \cong \K \big(\widehat{\ell^2}(G) \big) \otimes^X \K(A) \cong \K \big(\widehat{\ell^2}(G) \otimes^X A \big)
= \K\big(\widehat{\ell^2}(G,A) \big)
\end{equation}
as used in \cite[Proposition 3.2]{meyer}. This proposition goes essentially through unchanged
but uses also this lemma by Mingo and Phillips \cite{mingophillips}.

\begin{lemma}[Cf. Lemma 2.3 of \cite{mingophillips}]
If $\cale_1$ and $\cale_2$ are $G$-Hilbert $A$-modules which are isomorphic as Hilbert $A$-modules then $\widehat{\ell^2}(G,\cale_1)$
and $\widehat{\ell^2}(G,\cale_2)$ are isomorphic as $G$-Hilbert $A$-modules.
\end{lemma}

\begin{proof}
Let $u \in \call(\cale_1,\cale_2)$ be a unitary operator.
Then it can be checked that $V: \widehat{\ell^2}(G,\cale_1) \rightarrow  \widehat{\ell^2}(G,\cale_2)$
given by $V(\varphi_g \otimes \xi) := \varphi_g \otimes g u g^{-1}(\xi)$
defines an isomorphism of $G$-Hilbert $A$-modules. Note that $V$ is defined like $F'$ in Lemma
\ref{lemma31meyer}, so we can take the equivariance proof from there. For the inner product
note that $\langle \varphi_g , \varphi_h \rangle = \sum_{f \in F} f$ for a finite set $F \subseteq E$
with $fg = f h$ and $f \le g g^* h h^*$ by Lemma \ref{lemmaEcontinuity}, so that
\begin{eqnarray*}
\langle V(\varphi_g \otimes \xi),V(\varphi_h \otimes \eta)\rangle
&=& \sum_{f \in F} f \otimes \langle fgug^*f(\xi),f h u h^* f(\eta) \rangle\\
&=& \langle \varphi_g \otimes \xi,\varphi_h \otimes \eta\rangle.
\end{eqnarray*}
\end{proof}

The last lemma implies also the validity of an literally identical version of \cite[Theorem 2.4]{mingophillips}
($L^2(G,\cale)^\infty \cong L^2(G,A)^\infty$ $G$-equivariantly)
in our setting by the same proof.

In \cite[Lemma 4.3]{meyer} some homotopy results with $\calf^\infty$ are recalled.
The canonical proofs, using $L^2([0,1])$ (see \cite[Lemma 1.3.7]{thomsenbook}) work also inverse semigroup equivariantly.
In \cite[Lemma 4.4]{meyer} we note that we have to replace $\big( g(F)-F \big ) \phi(a)$ by
$\big(g(F) - g g^{-1}(F) \big) \phi(a)$.
We recall that $g g^{-1}$ is in the center of $\call(\cale)$ so that $\cale' := J \cdot \cale$ is $G$-invariant
because $g(J \cdot \cale)= g(J) \cdot g(\cale) \subseteq \cale'$. Everything goes through unchanged.
%

Section 5.1 in \cite{meyer} can be ignored since we do not need it.
In \cite[Section 5.2]{meyer} we have to replace $Q A:= A * A$ by the compatible
free product $Q A:=A *^X A$ by identifying $e(a) * b$ and $a * e(b)$ in $A * A$ for all $a,b \in A$ and $e \in E$.
Because of this identification, the diagonal action $g(a_1* \cdots *a_n) :=g(a_1) * \ldots * g(a_n)$ turns $QA$ to a $G$-algebra.
The kernel of the canonical $G$-equivariant $*$-homomorphism $A *^X A \rightarrow A$ is denoted by $q(A)$.

\begin{definition}	    \label{defintionKGNA}
{\rm
For a $G$-algebra $A$ we define
\begin{eqnarray*}
\K(G \N) A &:=& \K \big (\ell^2(\N) \otimes \big (\widehat{\ell^2}(G) \otimes^X A \big )\big )
\;\cong\; \K \big( (L^2(G,A))^\infty \big)
\end{eqnarray*}
(by $\cale^\infty := \ell^2(\N) \otimes \cale$ in \cite[Section 2.1.1]{meyer}).
(Confer also (\ref{isol2a}).)
}
\end{definition}

In accordance to the rules of Definition \ref{defHilbertCXmodules}
we may also write
$\K(G \N) A = \K \big ( C_0(X)^\infty \otimes^X \big (\widehat{\ell^2}(G) \otimes^X A \big )\big )$.

%

In the last paragraph of the proof of \cite[Proposition 5.4]{meyer}
one rewrites a special Kasparov triple $(\cale,\phi,F)$ as the Kasparov triple
$(\cale^+ \oplus \cale^+, \phi^+ \oplus \phi^-, P)$ by using the grading
on $\cale$ and identifying $\cale^-$ with $\cale^+$ via $F$;
$P$ is then the flip operator.
Here we need Definition \ref{defGequivariance} that $F$ commutes with the $G$-action such that
$F$ restricts to a $G$-equivariant Hilbert module isomorphism between $\cale^-$ and $\cale^+$,
and thus $\phi^-: A \rightarrow \call(\cale^+)$ is $G$-equivariant.





\begin{definition}  
{\rm
For $G$-algebras $A$ and $B$ set $[A,B]_s:= [\K(G\N)A,\K(G\N)B]$, where $[A,B]$ denotes the homotopy group of $*$-homomorphisms from $A$ to $B$.
Denote by $[C^*_G]_s$
the category of separable $G$-algebras as objects and morphism sets $[A,B]_s$ between objects $A$ and $B$.

}
\end{definition}

\begin{definition}    \label{defStability}
{\rm
A functor $F: C^*_G \rightarrow \calc$ into a category $\calc$ is called {\em stable}
iff the map $F(\K(\calh) A) \rightarrow F(\K(\calh \oplus \calh')A)$
induced by the inclusion $\calh \subseteq \calh \oplus \calh'$
is an isomorphism for all countably generated $G$-Hilbert $C_0(X)$-modules $\calh,\calh'$ and
all separable $G$-algebras $A$.
}
\end{definition}


Note that in \cite[Proposition 6.1]{meyer} $\C \oplus L^2(G \N)$
has to be replaced by $C_0(X) \oplus L^2(G \N)$.

\begin{proposition}[Cf. Proposition 6.3 of \cite{meyer}]
The canonical functor $C^*_G \rightarrow KK^G$ is a split exact stable homotopy functor.
\end{proposition}

\begin{proof}
We only remark stability and may prove this like in \cite[Lemma 3.1]{thomsen}.
Consider $\calh$ and $\calh'$ as in Definition \ref{defStability}, and
prove that the two cycles $(\iota, \K(\calh \oplus \calh'),0) \in KK^G(\K(\calh),\K(\calh\oplus \calh'))$
($\iota$ induced by the inclusion $\calh \subseteq \calh \oplus \calh'$)
and
$(id, \K(\calh \oplus \calh') p,0) \in KK^G(\K(\calh \oplus \calh'),\K(\calh))$ are inverses to each other,
where $p \in \call(\calh \oplus \calh')$ is the canonical projection onto the first factor $\calh$,
because $\K(\calh \oplus \calh') p \otimes_{\K(\calh)} \K(\calh \oplus \calh')
\cong \K(\calh \oplus \calh')$ via $a \otimes b \mapsto a b$.
We apply then the compatible version $\tilde \tau_A$ of Definition \ref{defTau} to these isomorphisms,
where $\otimes$ is replaced by the compatible tensor product $\otimes^X$,
to get isomorphisms with $\otimes^X A$.
%
%
%
\end{proof}


%
Since we did not observe every detail of \cite{meyer} - and this 
concerns particularly
\cite[Proposition 5.4]{meyer}, where we mainly only
focused on obvious differences 
between groups and inverse semigroups - the cautious reader should
view the following theorem as a conjecture!

\begin{theorem}[Adaption of Theorem 6.5 of \cite{meyer}, Conjecture]   \label{adaptionth65}
Assume that $G$ is $E$-continuous.
Let $A$ and $B$ separable (ungraded) $G$-algebras.
Define $q_s A := q(\K(G\N)A)$.
The canonical functor $C^*_G \rightarrow KK^G$ factors through a functor
$\sharp :[C^*_G]_s \rightarrow KK^G$.
There is a morphism $\pi_A^s \in [q_s A,A]_s$ (see \cite{meyer}), such that $\sharp(\pi_A^s) \in KK^G(q_s A,A)$
is invertible.
Then the map
$$\Delta:[q_s A,q_s B]_s \rightarrow KK^G(A,B), \qquad \Delta(f) = \sharp(\pi_B^s) \circ \sharp(f) \circ {\sharp(\pi_A^s)}^{-1}$$
is a natural isomorphism.
Hence the Kasparov product on $KK^G$ corresponds to the composition of homomorphisms.
\end{theorem}

By composing the functor $\Delta$ with the canonical functor $KK^G \rightarrow IK^G$ we
see that we can rewrite morphisms in $IK^G(A,B)$ which are represented by compatible cycles
also as $*$-homomorphisms in $IK$-theory.

%

\section{$\widetilde{KK}^G$ is a triangulated category}

\label{sectionTriangulated}

In this Section we recall the facts which show that $\widetilde{KK}^G$ is a triangulated category.
Everything from groups $G$ to inverse semigroups $G$ goes literally and canonically through
and needs no adaption,
the only exception from this being axiom (TR1) which is essentially Theorem \ref{adaptionth65}.
Actually we shall work with a slightly different category, the category $\widetilde{KK}^G$, rather than
the category $KK^G$ as we might expect. However, both categories are equivalent.

\begin{definition}
{\rm
Define $\widetilde{KK}^G$ (see \cite[Section 2.1]{meyernest})
to be the category where the objects are pairs $(A,n)$ for all separable
$G$-algebras $A$ and $n \in \Z$, and the morphism set between two objects $(A,n)$ and $(B,m)$ is defined to be
\begin{eqnarray*}
\widetilde{KK}^G\big((A,n),(B,m) \big) &:=& \lim_{p \in \N} KK^G(\Sigma^{n+p} A, \Sigma^{m+p} B).
\end{eqnarray*}
The maps in the direct limit are the maps $\tau_{C_0(\R)}$ and of course we require $n+p,m+p \ge 0$.
The composition of the morphisms is canonically via the Kasparov product.
}
\end{definition}

By Bott periodicity $\tau_{C_0(\R)}$ is an isomorphism, and so we may omit the direct limit.
However, it is needed at least to make desuspension, defined next.

\begin{definition}
{\rm
Define a {suspension functor} $\Sigma$ from $\widetilde{KK}^G$ to $\widetilde{KK}^G$
by $\Sigma(A,n) := (A,n+1)$ and $\Sigma(x):= \tau_{C_0(\R)}(x) \in KK^G(\Sigma^{n+p+1}A,\Sigma^{m+p+1}B)
\subseteq \widetilde{KK}^G \big((A,n+1),(B,m+1)\big)$
for all $x \in KK^G(\Sigma^{n+p}A,\Sigma^{m+p}B)
\subseteq \widetilde{KK}^G \big((A,n),(B,m)\big)$.
}
\end{definition}

The desuspension functor $\Sigma^{-1}$ on $\widetilde{KK}^G$ is defined to precisely reverse the functor $\Sigma$,
and we have $\Sigma \circ \Sigma^{-1} = \Sigma^{-1} \circ \Sigma = id_{\widetilde{KK}^G}$,
so $\Sigma$ is an isomorphism of categories.
The canonical map $KK^G \rightarrow \widetilde{KK}^G$ sending $A$ to $(A,0)$ 
is an equivalence of categories.
Indeed, by Bott periodicity, $KK^G(\Sigma^{2n}A,B) \cong KK^G(A,B)$, every element $(A,n)$ is isomorphic to some $(B,0)$ in $\widetilde{KK}^G$.
(We have $(A,2n) \cong (A,0)$ and $(A,2n+1) \cong (\Sigma A,0)$.)
Most of the time it is sufficient to think of $\widetilde{KK}^G$ just as $KK^G$.

Having now a suspension functor $\Sigma$, we further need distinguished triangles to turn $\widetilde{KK}^G$ into a triangulated
category.

\begin{definition}    \label{defMappingConeTriangle}
{\rm
Let $A$ and $B$ $G$-algebras. Then to an equivariant $*$-homomorphism $f:A \rightarrow B$
we associate the {\em mapping cone} (cf. \cite[Section 2.1]{meyernest}), which is the $G$-algebra
\begin{eqnarray}   \label{defmappingcone}
\mbox{cone}(f)  &:=&  \{ (a,b) \in A \times C_0 \big((0,1],B \big)\,|\, f(a) = b(1)\},
\end{eqnarray}
and the {\em mapping cone triangle}, which is the sequence of equivariant $*$-homomorphisms
\begin{equation}   \label{mappingconetriangle}
\xymatrix{\Sigma B \ar[r]^\iota & \mbox{cone}(f) \ar[r]^\epsilon & A \ar[r]^f  & B },
\end{equation}
where $\iota$ is the canonical inclusion (setting the coordinate $a$ to zero) and
$\epsilon$ is the canonical projection onto $A$.
}
\end{definition}

\begin{definition}  \label{defExactTriangle}
{\rm
A diagram $\Sigma B' \rightarrow C' \rightarrow A' \rightarrow B'$
in $\widetilde{KK}^G$ is called an {\em exact triangle} (see \cite[Section 2.1]{meyernest}) if it is isomorphic to a mapping cone
triangle (\ref{mappingconetriangle}) in $\widetilde{KK}^G$, that is, there exists an equivariant $*$-homomorphism $f:A \rightarrow B$
and a commutative diagram
$$
\xymatrix{\Sigma B \ar[r]^\iota \ar[d]^{\Sigma \beta} & \mbox{cone}(f) \ar[r]^\epsilon \ar[d]^{\gamma} & A \ar[r]^f \ar[d]^{\alpha} & B \ar[d]^{\beta}\\
\Sigma B' \ar[r] & C' \ar[r] & A' \ar[r] & B' }
$$
where $\alpha,\beta$ and $\gamma$ are isomorphisms and the suspension $\Sigma \beta$ of $\beta$ is of course also an isomorphism.
}
\end{definition}

For convenience of the reader we recall the definition of extension triangles, which are exact triangles
in the sense of Definition \ref{defExactTriangle}, and which are technically used in the proof that $\widetilde{KK}^G$
is a triangulated category.

\begin{definition}[Definition 2.3 in \cite{meyernest}]    \label{defExtensionTriangle}
{\rm
Let $\cale:0 \rightarrow A \stackrel{i}{\rightarrow} B \stackrel{p}{\rightarrow} C \rightarrow 0$
be an extension of $G$-algebras and associate to it the commuting diagram (without the indicated map $\mu$)
\begin{equation}  \label{diagramExtensiontriangle}
\xymatrix{
\Sigma C \ar[r]^\mu \ar[d]^{id}  & A \ar[r]^i \ar[d]^\alpha  & B \ar[r]^p \ar[d]^{id}  & C \ar[d]^{id}  \\
\Sigma C \ar[r]^\iota  & \mbox{cone}(p) \ar[r]^\epsilon  & B \ar[r]^p  & C  \\
}
\end{equation}
where $\mbox{cone}(p) \subseteq B \times C_0((0,1],C)$,
$\iota(c):= (0,c)$, $\epsilon(b,c):= b$ and $\alpha(a):= (i(a),0)$
for all $c \in C_0((0,1),C)$, $b \in B$ and $a \in A$.
The extension $\cale$ is called {\em admissible} if $\alpha$ is an isomorphism in $\widetilde{KK}^G$.
In this case we have an obvious 
morphism $\mu:= \alpha^{-1} \circ i$ which makes the diagram
(\ref{diagramExtensiontriangle}) to an isomorphism of exact triangles in $\widetilde{KK}^G$
in the sense of Definition
\ref{defExactTriangle}
(since the second line is obviously a mapping cone triangle), and in this case we call the first line
of (\ref{diagramExtensiontriangle}), which is an exact triangle, also the {\em extension triangle} of $\cale$.
%
}
\end{definition}


We shall not need the following lemma but state it as an interesting observation in its own.
It is proved like in the last paragraph of \cite[Section 2.3]{meyernest}.

\begin{lemma}[Section 2.3 in \cite{meyernest}]
Every exact triangle is isomorphic to an extension triangle in $\widetilde{KK}^G$.
\end{lemma}


\begin{proposition}[Proposition 2.1 and Appendix A of \cite{meyernest}]		\label{propositiontriangulatedcategory}
Suppose that $G$ is $E$-continuous
and Theorem \ref{adaptionth65} is correct.
The category $\widetilde{KK}^G$ endowed with the translation functor $\Sigma^{-1}$ (the suspension functor
in a triangulated category) and exact triangles from Definition \ref{defExactTriangle}
is a triangulated category.
\end{proposition}

\begin{proof}
One of the axioms of an triangulated category, the axiom (TR1) of \cite{neemanbook}, requires
that every morphism $f:A \rightarrow B$ in $\widetilde{KK}^G$ fits into an exact triangle
$\Sigma B \rightarrow C \rightarrow A \stackrel{f}{\rightarrow} B$.
If $f$ is actually a $*$-homomorphism then we may take the mapping cone triangle
as an exact triangle (see Definitions \ref{defMappingConeTriangle} and \ref{defExactTriangle}).
Given a general morphism $f \in KK^G(A,B)$ we rewrite it as the image of the map $\Delta$
of Theorem \ref{adaptionth65}, that is $f = x \circ g \circ y$, where $g:q_s A \rightarrow q_s B$ is an equivariant $*$-homomorphism, and $x \in {KK}^G(q_s A,A)$ and $y \in {KK}^G(B,q_s B)$
are isomorphisms in ${KK}^G$, and take the mapping cone triangle for $g$.

The rest of the axioms are proved in Appendix A of \cite{meyernest} directly by using canonical
equivariant $*$-homomorphisms including homotopies, and 
extension triangles as in Definition
\ref{defExtensionTriangle}. This canonical proof goes literally through also in our setting.
%
\end{proof}

Like in \cite{meyernest}, in the remainder of this paper we sloppily do not distinguish between the equivalent categories $KK^G$ and $\widetilde{KK}^G$ and shall work practically exclusively with $KK^G$.
%


\section{Some lemmas with restriction and induction}

\label{sectionLemmasResInd}

In this section we present a mix of lemmas which deal with restriction and induction functors
and which might be of independent interest.
They may be particularly interesting as they handle equivalence relations on inverse semigroups,
which are seldom considered and difficult for inverse semigroups as compared to groups,
and projections which do not appear in groups at all.
We shall often 
leave out notating the restriction functor $\mbox{Res}_H^G$
where it is obviously there
for better readability.

The following Lemma \ref{lemmatechnicalDirac} prepares Lemma \ref{lemmaKKLGsplit}.
They deal with expressions where induction and restriction functors come together.


\begin{lemma}    \label{lemmatechnicalDirac}
Let $U' \subseteq G$ a finite subinverse semigroup of $G$ and
$U$ its associated finite groupoid.
Let $L \subseteq G$ be a subinverse semigroup of $G$. Let $D$ be $G$-algebra.
Let $g \in G_U$ (that is, $g=g_0 u_0$ for some $g_0 \in G$ and $u_0 \in U^{(0)}$).
Define $L'$ as the subinverse semigroup of $G$ generated by $L \,\cup \,g_0 E(U') g_0^{*}$
and set
$M:=(g g^* L {g g^{*}} \, \cap \, g U g^{*}) \backslash \{0\} \subseteq \tilde G$.
Then we have an isomorphism of $L$-algebras
\begin{eqnarray*}
&&\theta: 
\mbox{Ind}_{M}^{L'} \mbox{Res}_G^M(D)  \longrightarrow
\{f \in \mbox{Ind}_{U}^G \mbox{Res}_G^U (D) \,|\, \mbox{$f$ has carrier in $L g U \cap G_U$}\}
\end{eqnarray*}
via $\theta(f)(lg u) = u^{*} g^* (f(lgg^{*}))$
for all $f \in 
\mbox{Ind}_{M}^{L'}(D)$, $l \in L$ and $u \in U$.

\end{lemma}

\begin{proof}
We may write $g = g_0 u_0$ for some $g_0 \in G$ and $u_0 \in U^{(0)}$,
and note that $g_0^* g_0 \ge u_0$ and $g^* g = u_0$.
Note that $M \subseteq \widetilde{L'}$ since $g g^* = g_0 u_0 g_0^*$ can be expressed in $\widetilde{L'}$. 
Of course, every element of $M$ has source and range projection $g u g^* g u^* g^* = g g^* \in \widetilde{G}$,
so $M$ is a subgroupoid (or even subgroup) of $\widetilde{L'}$.
If there is $l \in L$ such that $l^* l \ge g g^*$ then the indicated image of $\theta$ is nonempty,
if and only if $g l^* l g^* = g g^* \in M$, if and only $M$ is nonempty, the case we are considering now, because
otherwise $\theta$ is, correctly, the empty function.
%
%
%
Every element $l' \in L'$ may be written in the form
\begin{equation}   \label{linel}
l' = (g_0 u_1 g_0^*) l_1 (g_0 u_2 g_0^*) l_2 (g_0 u_3 g_0^*) \ldots l_n (g_0 u_n g_0^*) = l p
\end{equation}
for some $u_i \in E(U')$, $l_i, l \in L$ and $p \in E(L')$.
Then an element is in $(L' )_M \subseteq \tilde G$ 
if and only if it is of the form $l' g g^*$ with $l' \in L'$ and $l'^* l' \ge g  g^*$.
We may write $l' g  g^* = l p (g  g^*) = l g g^*$ by (\ref{linel}),
and because the source projection of $l' g g^*$ is $g g^*$, we also have $l^* l \ge g g^*$.
Hence we have obtained
\begin{equation}    \label{eqlsm}
(L')_M = \{ l g  g^*\, \in \tilde G |\, l \in L, \, l^* l \ge g g^* \}.
\end{equation}

To show that $\theta$ is well defined, consider an ambiguously 
represented element $l g u = l' g u' \in L g U \cap G_U$ for
$l,l' \in L$ and $u,u' \in U$. Notice that $l^* l, {l'}^* l' \ge g g^*$ (because of $G_U$),
and that source and range projections of $u$ and $u'$ are the same.
Thus $g u u'^* g^* = l^* l' g g^*$ is in $M$.
Hence
\begin{eqnarray*}
&& \theta(f)(l' g u') \;= \;u'^*g^*\big(f(l'g g^*)\big) \;=\;  u'^*g^*\big(f(l g u u'^* g^*)\big)   \\
&=&  u'^*g^* (g u u'^* g^*)^*\big (f(l g ) \big) \; = \;  u^*g^* \big (f(l g ) \big) \; = \;
\theta(f)(l g u).
\end{eqnarray*}

Injectivity of $\theta$ follows from $g u (\theta(f)(lgu))= g g^*(f(lg g^*)) = f(l g g^* g g^*)$
(because $g g^* \in M$)
and identity (\ref{eqlsm}). To check surjectivity of $\theta$, write a given $j \in \mbox{Ind}_U^G(D)$ with carrier
in $LgU \cap G_U$ as $j=\theta(f)$ for the 
$f \in \mbox{Ind}_{M}^{L'}(D)$
determined by $f(l g g^*) := g(j(lg))$ for all $l \in L$ (confer also (\ref{eqlsm})).
In verifying $L$-invariance of $\theta$, we compute
\begin{eqnarray*}
&& \theta \big(h(f) \big)(l g u_0) \; =\; g^* \big (h(f) (l g g^*) \big) \;=\;
g^* \big( f(h^* l g g^*) \big) \, [ h h^* \ge l g g^* l^*]    \\
&=& \theta(f)(h^* l g)\, [ h h^* \ge l g g^* l^*] \; = \; h\big(\theta(f)\big) (l g u_0)
\end{eqnarray*}
for all $h,l \in L$.
\end{proof}

\begin{lemma}    \label{lemmaKKLGsplit}
Let $H'$ a finite subinverse semigroup of $G$ and $H$ its associated finite subgroupoid of $\tilde G$. Let $L$ be a subinverse semigroup of $G$.
Let $D$ be a $G$-algebra.
Then there is an $L$-equivariant $*$-isomorphism  
\begin{eqnarray*}
\mbox{Res}_G^L \mbox{Ind}_H^G 
\mbox{Res}_G^H (D) &\cong& \bigoplus_{g \in J} \mbox{Res}_{L'_g}^L
\mbox{Ind}_{M_g}^{L'_g} 
\mbox{Res}_G^{M_g}(D),
\end{eqnarray*}
where $J \subseteq G$ is a 
subset and $M_g$ is the set $M$ of Lemma \ref{lemmatechnicalDirac} for $U':= H'$.
\end{lemma}

\begin{proof}
Say that two elements $g,g' \in G_H$ are $L$-equivalent if $l g = g'$ for some $l \in L$ with $l^* l \ge g g^*$.
This relation is reflexive as $1 \in L$, symmetric because $l^* l g = g = l^* g'$ and $l l^* \ge l g g^* l^*
= g' g'^*$, and transitive because $l g= g' = l'' g''$ implies $g = l^* l'' g''$ and $l''^* l l^* l''
\ge l''^* l g g^* l^* l'' = l''^* l'' g g^* l''^* l'' = g g^*$.
Similarly, two elements in $g, g' \in G_H$ are said to be $L,H$-equivalent if $l g h = g'$ for some $l \in L$
with $l^* l \ge g g^*$ and some $h \in H$, and this is also an equivalence relation.
Its equivalence classes are exactly of the form $L g H \cap G_H \subseteq G_H$
(the intersection taken in $\tilde G$)
for all $g \in G$

For every $g \in G$
apply Lemma \ref{lemmatechnicalDirac} for $U':= H'$, and denote $\theta$ of Lemma \ref{lemmatechnicalDirac}
more precisely
by $\theta_g$, the image of $\theta_g$ by $F_g$, $M$ by $M_g$
and $L'$ by $L'_g$. Note that $F_g$ is a $L$-invariant $C^*$-subalgebra
of $\mbox{Ind}_H^G(D)$.
Choose from every $L,H$-equivalence class exactly one representative $g \in G$ and denote their
collection by $J \subseteq G$. (We remove those $g$ for which $F_g$ is empty.) Of course, we have a canonical $*$-isomorphism of $L$-algebras
\begin{eqnarray*}
\mbox{Res}_G^L \mbox{Ind}_H^G \mbox{Res}_G^H(D) &\cong& \bigoplus_{g \in J} F_g  \;\;\cong \;\; \bigoplus_{g \in J} \mbox{Res}_{L'_g}^L
\mbox{Ind}_{M_g}^{L'_g} \mbox{Res}_G^{M_g}(D),
\end{eqnarray*}
the last isomorphism being the one induced by the $\theta_g$s.
\end{proof}

The idea of the next lemma is
to get rid off the $\mbox{Res}_{L'_g}^L$-term appearing in the last lemma,
where $L$ and $L'_g$ distinguish only by 
projections which could not appear in a group.

\begin{lemma}   \label{lemmaIsoLLprime}
Let $L \subseteq G$ be a finite subinverse semigroup and $P \subseteq G$ a subset of projections.
Let $L' \subseteq G$ denote the subinverse semigroup generated by $L \cup P$.
Assume that $L'$ is $E$-unitary.
Let $A$ be a finite dimensional, commutative $L'$-algebra.
Let $B$ be a $L$-algebra. Then there exists an $n\ge 1$
and a $L'$-action on a (quite canonical) subalgebra $B' \subseteq B^n$ such that
$$KK^{L}(Res_{L'}^L\,A,B) \cong KK^{L'}(A,B').$$
The assignment $B \mapsto B'$ commutes canonically with all (infinite) direct sums.
\end{lemma}

\begin{proof}
Note that $L'=\{l p \in L'|\, l \in L,\, p \in E(L')\}$. Similarly, writing $W:=\widetilde{L'}$,
$W = \{l p \in W|\, l \in L,\, p \in E(W) \}$.
%
Let $\alpha$ denote the $L'$-action on $A$ and $\beta$ the $L$-action on $B$.
Note that $A$ is of the form $\C^n =C_0(\{1,\ldots,n\})$ and so the $L'$-action can only cancel 
or permute the factors $\C$.
Consider the finite set $\alpha(E(W)) \subseteq \call(A)$ of projections,
which is already a refined set of projections,
and enumerate by $(p_{i,j})_{1 \le i \le m,1 \le j \le n_i}$ all their minimal projections,
where $p_i := \sum_{j=1}^{n_i} p_{i,j}$ denotes the minimal projections
of the smaller projection set $\alpha(E(\tilde L))$. 
Choose a selection (lift) $\sigma:\{p_{i,j}\} \rightarrow W$ such that $\alpha \circ\sigma = id$,
and write $q_{i,j}:= \sigma(p_{i,j})$ for simplicity.
Also, denote by $q_1,\ldots, q_m \in W$ the minimal projections of $E(\tilde L)$.

Let us be given a cycle $(\pi,\cale,T)$ in $KK^{L}(A,B)$. We want to mirror the $W$-structure
of the $A$-side to the $B$-side. By a well known cut-down of a cycle, we may assume without loss of generality
that $\pi(1) =1_{\call(\cale)}$. Denote the $L$-action on $\cale$ by $\gamma$.
Set $B_i:= \beta(q_i) B \subseteq B$ for $1 \le i \le m$. Note that $B
\cong B_1 \oplus \ldots \oplus B_m$.
(Also observe that $\cale$ has an analog, associated decomposition $\cale = \pi(p_1(1)) \cale \oplus \ldots \oplus \pi(p_m(1))\cale$
by $L$-equivariance of $\pi$.)
Define $B':=\bigoplus_{i=1}^m B_i^{n_i} = \bigoplus_{i=1}^m \bigoplus_{j=1}^{n_i} B_i$.
Denote these summands by $B_{i,j}$.
We want to define a cycle $(\pi',\cale',T')$ in $KK^{L'}(A,B')$.
Let $\cale'$ denote an identical copy of $\cale$ as a graded vector space.
We define a $B'$-valued inner product on $\cale'$ by
$$\langle \xi,\eta \rangle_{\cale'}:= \oplus_{i,j}
\big \langle \pi\big(p_{i,j}(1)\big)\xi, \pi\big(p_{i,j}(1)\big) \eta \big \rangle_{\cale} \quad \in \; B'= \oplus_{i,j} B_{i,j}$$
for all $\xi,\eta \in \cale'$, and the $B'$-module multiplication on $\cale'$ 
by $\xi (\oplus_{i,j} b_{i,j}) := \sum_{i,j} (\pi((p_{i,j}(1)))\xi)
b_{i,j}$ (the last $b_{i,j}$ regarded in $B_i$). Define a $L'$-action $\gamma'$ on $\cale'$ by $\gamma'(l
p) := \gamma(l) \pi \big (\alpha(p)(1)\big )$
for all $l \in L$ and $p \in E(L')$. 
Because $L'$ is $E$-unitary, the presentation $l p$ with $p \le l^* l$ is unique and thus $\gamma'$ well-defined.
Define a $W$-action $\beta'$ on $B'$ by $\beta'(l p)(\oplus_{i,j} b_{i,j}) = \oplus_{i,j} 1_{\{i=i_1\}} 1_{\{j=j_1\}}\beta(l)(b_{i_0,j_0})$
if $\alpha(p)=p_{i_0,j_0}$ and $\alpha(lp)$ has source projection $p_{i_0,j_0}$ and range projection
$p_{i_1,j_1}$.
%
%
%
We extend this definition to a $W$-action by additivity, that is, $\beta'(\sum_{i,j} \lambda_{i,j} l q_{i,j}) := \sum_{i,j} \lambda_{i,j} \beta'(l q_{i,j})$ for $l \in L$ and $\lambda_{i,j} \in \{0,1\}$.
%
%
%

Noting that $\sum_{i,j} \pi(p_{i,j}(1)) = 1_{\call(\cale)}$, we may write $T$ in matrix form
$(T_{(i,j),(i',j')})_{(i,j),(i',j')}$.
Since $[T,\pi(p_{i,j}(1))] \in \calk(\cale)$, all off-diagonal elements of $T$ are compact operators
and so by a compact perturbation we may replace $T$ by its diagonal matrix $T'$ (canceling the off-diagonal terms
of $T$)
without changing the cycle, that is, $[(\pi,\cale,T)]=[(\pi,\cale,T')]$.
Note that the identical map $\call(\cale) \cap \mbox{diagonal matrices} \rightarrow \call(\cale')$
is an isomorphism, which restricts to a bijection
$\calk(\cale) \cap \mbox{diagonal matrices} \rightarrow \calk(\cale')$ because
$\pi(p_{i,j}(1)) \cale \cong \pi(p_{i,j}(1)) \cale'$ 
for all $i,j$.
We set $\pi':= \pi$. The desired cycle in $KK^{L'}(A,B')$ is $(\pi',\cale',T')$.

Let us reversely be given a cycle $(\pi',\cale',T')$ in $KK^{L'}(A,B')$.
Define $\pi:=\pi'$, $T:= T'$ and $\cale$ an identical copy of $\cale'$ as a graded vector space.
(Note that $\cale \cong \oplus_{i,j} \pi(p_{i,j}(1)) \cale$ corresponding to $B'$ by $L'$-equivariance
of $\pi$.) 
Set
$$\langle \xi,\eta \rangle_{\cale} := \bigoplus_{i=1}^m \sum_{j=1}^{n_i} \big \langle \pi \big (p_{i,j}(1) \big) \xi, \pi \big (p_{i,j}(1) \big ) \eta \big \rangle_{\cale'} \quad \in \; B = B_1 \oplus \ldots \oplus B_m$$
for all $\xi,\eta \in \cale$, the $B$-module product on $\cale$ by 
$\xi (\oplus_i b_i) := \sum_i \sum_{j} \pi(p_{i,j}(1)) \xi b_i$ (the last $b_i$ regarded in $B_{i,j}$),
and the $L$-action on $\cale$ to be the restriction of the $L'$-action on $\cale'$.
It is easy to see that both constructed assignments $(\pi,\cale,T') \leftrightarrow (\pi',\cale',T')$
are reverses to each others.
The detailed, tedious verifications we left out in this proof are left to the reader.
\end{proof}

The next lemma deals with the question how to remove $\mbox{Res}_{G}^{pG}$.

\begin{lemma}   \label{lemmaKKpL}
Let $p \in G$ be a projection in the center.
Then $KK^{p G}(\mbox{Res}_{G}^{pG} \,A,\mbox{Res}_{G}^{pG} \, B) \cong KK^G(p A,p B) \cong KK^G(p A,B)
\cong KK^G(A, pB)$.
\end{lemma}

\begin{proof}
The first isomorphism is just the identity on cycles; a cycle $(\cale,T)$ in $KK^G(p A,p B)$
degenerates to $(p \cale,p T)$; a $pG$-action extends to a $G$ action by $g \mapsto pg$.
Also recall that $\mbox{Res}_G^{pG}(A) = p A$.
For the second isomorphism we decompose $B \cong p B \oplus (1-p) B$
and note that $KK^G(p A,(1-p)B) = 0$ since $p(a) \xi (1-p) (b) = 0$ for $a \in A,\xi \in \cale $ and  $b \in B$,
where $(\cale,T)$ is a cycle.
%
\end{proof}

The next lemma is similar to the fact that the $K$-theory group $KK(\C,B)=K(B)$ is countable.
It is immediately evidently true in $IK$-theory by
the Green--Julg isomorphism $IK^H(\C,A) \cong K(A \rtimes H)$ in \cite{burgiGreenjulg}.

\begin{lemma}  \label{lemmaCountableKKGC}
For all compact subinverse semigroups $H \subseteq G$
$KK^H(\mbox{Res}_G^H  \, \C,B)$ is countable for all $B \in KK^G$
and commutes with countable direct sums in the variable $B$.
%
%
\end{lemma}

\begin{proof}
%
%
Let $f: \C \rightarrow C_0(X_H)$ be the map $f(1)= 1_e$, where $e$ denotes the minimal projection in
$E(H)$, so is also in $X_H$. Reversely, let $p: C_0(X_H) \rightarrow \C$ be the projection $p(1_e) = 1$.
Both $f$ and $p$ are $G$-equivariant $*$-homomorphisms, because $g(1_e) = 1_{g e g^*} = 1_e$ since
$g e g^*$ is both in $X_H$ and in $E(H)$, so must be $e$ again.
The map $f^*: KK^H(C_0(X_H),B) \rightarrow KK^H(\C,B)$ is surjective and $p^*$ is injective
because
$f^* p^* = (p f)^* = id$.
Hence, noting that the $K$-theory of a separable $C^*$-algebra is countable, $KK^H(\mbox{Res}_G^H  \, \C,B)$ is countable since it is the image of $f^*$ of the countable abelian group
\begin{eqnarray}   \label{greenj}
KK^H (C_0(X_H), \mbox{Res}_G^{H}\,B) & \cong &  K \big(\mbox{Res}_G^{H}(B) \widehat \rtimes H \big),
\end{eqnarray}
where this is essentially the Green--Julg isomorphism for groupoids, see Tu \cite[Proposition 6.25]{tu1999nov},
or directly apply \cite[Corollary 5.4]{burgiKKrDiscrete}.
%
%
Both diagrams
\begin{equation}   \label{commdiagram}
\xymatrix{
\bigoplus_i KK^H(C_0(X_H), B_i) \ar[r]
\ar[d]^{\bigoplus_i f^*}  &  KK^H \big(C_0(X_H),\bigoplus_i B_i \big)
 \ar[d]^{f^*} \\
\bigoplus_i KK^H(\C, B_i) \ar[u]^{\oplus_i p^*}  \ar[r]
 &  KK^H \big(\C,\bigoplus_i B_i \big)
\ar[u]^{p^*}
}
\end{equation}
commute (one with $f^*$ and another with $p^*$) and because the first line is an isomorphism because of (\ref{greenj}) ($K$-theory respects direct sums),
the second line is also one.
%
\end{proof}



\section{Some motivating specialized results}

\label{sectionMotivationDiracmorphism}


In this section we shall prove some specialized results with induction and restriction on the way for
proving the existence of a Dirac morphism.
In this section we shall work with the restriction functor, and only in the next section with the correct right adjoint functor
for induction, and hence this section will lead to a dead end.
Nevertheless the results of this section will 
motivate the assumptions met in the next section for hypothetically
proving the existence of a Dirac morphism and a Baum--Connes map.
 
%



\begin{definition}
{\rm
Set
$$\calc \cali_1:= \{\mbox{Ind}_{H_n}^G \mbox{Res}_G^{H_n}
\ldots \mbox{Ind}_{H_1}^G \mbox{Res}_G^{H_1} (\C) \,| \, H_i \subseteq G \,\mbox{compact subinverse s.},
\, n \ge 1\}.$$
}
\end{definition}

Considering for example an object in $\calc \cali_1$ for $n=3$, we may write it as
\begin{eqnarray}    \label{successiveInd1}
\mbox{Ind}_{H_3}^G \mbox{Res}_{G}^{H_3} \mbox{Ind}_{H_2}^G \mbox{Ind}_{H_1}^G\, \C &=&
\mbox{Ind}_{H_3}^G \bigoplus_{g \in J} \mbox{Res}_{L_g'}^{H_3} \mbox{Ind}_{M_g}^{L_g'} \mbox{Res}_{G}^{M_g} \mbox{Ind}_{H_1}^G\, \C
\end{eqnarray}
by an application of Lemma \ref{lemmaKKLGsplit}.
Go back to Lemma \ref{lemmatechnicalDirac} and
define $V \subseteq G$ to be the finite subinverse semigroup $g_0 U'' g_0^*$,
where $U'' \subseteq U'$ denotes the finite subinverse semigroup consisting of those
elements $u \in U'$ such that $u$ commutes with $u_0$ and $g u g^* \in M \cup \{0\}$.
Note that $M \subseteq \widetilde{V}$ since $E(U') \subseteq U''$ and so $g_0 u_0 g_0^* \in \widetilde{V}$.
Observe that 
$g g^*$ is in the center of $\tilde V$ and $\tilde V g g^* = V g g^* = M \cup \{0\}$.
Write $V_g$ for the $V$ of $M_g$.
Continue (\ref{successiveInd1}) with
\begin{eqnarray}
&=&
\bigoplus_{g \in J} \mbox{Ind}_{H_3}^G \mbox{Res}_{L_g'}^{H_3} \mbox{Ind}_{M_g}^{L_g'} \mbox{Res}_{V_g}^{M_g}
\mbox{Res}_{G}^{V_g} \mbox{Ind}_{H_1}^G \mbox{Res}_{H_1}^{G} \, \C  \label{successiveInd2} \\
&=&\bigoplus_{g \in J} \bigoplus_{h \in J_g} \mbox{Ind}_{H_3}^G  \mbox{Res}_{L_g'}^{H_3} \mbox{Ind}_{M_g}^{L_g'} \mbox{Res}_{V_g}^{M_g} \mbox{Res}_{L_{g,h}'}^{V_g} \mbox{Ind}_{M_{g,h}}^{L_{g,h}'} \mbox{Res}_{G}^{M_{g,h}}\, \C  \label{successiveInd}
\end{eqnarray}
by another application of Lemma \ref{lemmaKKLGsplit}.

 Note that every summand in
(\ref{successiveInd}) is of the form $\mbox{Ind}_{H_3}^G\, A$ for some finite dimensional, commutative $H_3$-algebra
$A$. (Because $(L')_M$ is finite by (\ref{eqlsm}).)
Similarly, by a successive $n$-fold application of Lemma \ref{lemmaKKLGsplit} write
$\mbox{Ind}_{H_n}^G  \ldots \mbox{Ind}_{H_1}^G \, \C$ as a countable
direct sum of $G$-algebras of the form $\mbox{Ind}_{H_n}^G \, A$ for some finite dimensional, commutative $H_{n}$-algebras $A$.

\begin{definition}		\label{definitionCI0}
{\rm
Varying over all $n \ge 1$ and $H_1,\ldots,H_{n} \subseteq G$, denote by $\calc \cali_0$ the countable collection
of all $G$-algebras of the form $\mbox{Ind}_{H_{n}}^G\, A$ as just described.
}
\end{definition}

From here we shall assume that $G$ is $E$-continuous, for $KK^G$ to be a triangulated
category in the sense of Proposition \ref{propositiontriangulatedcategory}.

\begin{definition}   \label{defLocSubcat}
{\rm
A subcategory $\cals$ of a triangulated category $\calt$ is called a {\em triangulated subcategory}
(see \cite[Section 4.5]{krause})
if it is nonempty, full, closed under suspension and desuspension, and, whenever for a given exact sequence
$A \rightarrow B \rightarrow C \rightarrow SC$ two objects of $\{A,B,C\}$ are in $\cals$ then also
the third one.
$\cals$ is also called {\em thick} (see \cite[Section 4.5]{krause}) if every retract (summand) of an object
in $\cals$ is also in $\cals$,
and {\em localizing} (see \cite[Section 6.2]{krause}) if it is thick
and closed under coproducts in $\calt$.
}
\end{definition}

\begin{definition}  \label{defGeneratedLocCat}
{\rm
For a class $\calg$ of objects in $\calt$ we write $\langle \calg \rangle$ for the smallest localizing
subcategory of $\calt$ containing $\calg$, cf. \cite[Section 2.5]{meyernest}.
}
\end{definition}

Note that in $KK^G$ coproducts are direct sums, and we only allow {\em countable} direct sums.
In the next definition we {\em assume} that the used results by A. Neeman hold true under this
countability restriction for coproducts (we have not checked this).

\begin{definition}[Conjecture]    \label{defRightadjointtoInd}
{\rm
Suppose that $G$ is $E$-continuous.
Fix a compact subinverse semigroup $H \subseteq G$.
Let $\calf_H$ denote the set of all finite dimensional, commutative
$H$-algebras which are compact objects of the category $KK^H$ in the sense of \cite[Definition 1.6]{neeman1996}.
(For instance, $\C \in \calf_H$ by Lemma \ref{lemmaCountableKKGC}.)
The set $\Sigma \calf_H \cup \calf_H$ is closed under suspension by Bott periodicity and consists of compact objects.
By \cite[Proposition 8.4.1]{neemanbook} it is a generating set
for $\langle \calf_H \rangle$.
Hence $\langle \calf_H \rangle$ 
%
is a compactly generated triangulated category
in the sense of \cite[Definition 1.7]{neeman1996}.
By Lemma \ref{lemmaInductionfunctorIntertwines} and \cite[Theorem 4.1]{neeman1996},
the restricted induction functor $\mbox{Ind}_{H}^L:\langle \calf_H\rangle \rightarrow KK^L$ has a right
adjoint functor $\mbox{Right}_L^H:KK^L \rightarrow \langle \calf_H\rangle$ for every subinverse semigroup $L \subseteq G$.

}
\end{definition}


\begin{corollary}   \label{corollaryCountableKKGI0}
Assume that $G$ is $E$-unitary, $E$-continuous and the functors $\mbox{Right}_H^L$ respect countable direct sums.
(In the worst case scenario, if $G$ is a group.)
Then for all $A \in \calc \cali_0$ $KK^G(A,B)$ is countable for all $B \in KK^G$
and commutes with countable direct sums in the variable $B$.
%
%
\end{corollary}



\begin{remark}   \label{remarkCI0}
{\rm
It appears natural that $KK^H(A,B)$ is countable and commutes with countable direct sums in $B$
for all finite subinverse semigroups $H \subseteq G$ and finite-dimensional, commutative $H$-algebras $A$.
(The K\"unneth theorem comes into mind, but is difficult even for $G = \Z/2$, see Rosenberg \cite{rosenberg2013}.)
But then the claim of Corollary \ref{corollaryCountableKKGI0} would follow alone from Definition \ref{definitionCI0}
and the assumption that $\mbox{Right}$ respects countable direct sums.
}
\end{remark}


\begin{proof}[Proof of Corollary \ref{corollaryCountableKKGI0}]
To demonstrate the proof of Corollary \ref{corollaryCountableKKGI0}, assume $A$ is one of the summands of (\ref{successiveInd}).
We go inductively from right to left in (\ref{successiveInd}).
The first algebra $A_{1}:= \mbox{Res}_{G}^{M_{g,h}}\, \C $
of (\ref{successiveInd}) satisfies the claim of Corollary \ref{corollaryCountableKKGI0} when replacing $A$ by $A_1$ by Lemma \ref{lemmaCountableKKGC}.
The next algebra $A_2:=\mbox{Ind}_{M_{g,h}}^{L_{g,h}'}\, A_1$ satisfies the claim of Corollary \ref{corollaryCountableKKGI0} because
now evidently $A_1 \in \calf_{M_{g,h}}$ and we assume that $\mbox{Right}^{M_{g,h}}_{L_{g,h}'}$
respects countable direct sums, whence $A_2$ satisfies the claim by
putting $\mbox{Ind}$ to the other side as $\mbox{Right}$,
cf. \cite[Theorem 5.1]{neeman1996}.
%
Going back how we deduced identity (\ref{successiveInd}) from Lemma \ref{lemmaKKLGsplit}, a check shows that both expressions
$\mbox{Res}_{L_g'}^{H_3}$ and $\mbox{Res}_{L_{g,h}'}^{V_g}$ of (\ref{successiveInd}) are of the form
$\mbox{Res}_{L'}^{L}$, where $L'$ and $L$ are the notions from Lemma \ref{lemmatechnicalDirac} and
additionally $L$ is finite.
%
But from Lemma \ref{lemmaIsoLLprime} we know that
\begin{equation}   \label{isoLLprimeAstart}
KK^L(\mbox{Res}_{L'}^{L} \,A_{2},B) \cong KK^{L'}(A_{2}, B').
\end{equation}
Since $A_{2}$ satisfies the assumption, $\mbox{Res}_{L'}^{L} A_{2} = \mbox{Res}_{L_{g,h}'}^{V_g} A_{2} =:A_{3}$
does it also because of (\ref{isoLLprimeAstart}).
%
Recall that $g g^*$ is in the center of $\widetilde{V_g}$ and $\widetilde{V_g} g g^* = M_g \cup \{0\}$.
(See before (\ref{successiveInd2}).)
%
Consequently 
we have
\begin{equation}   \label{eqMgpL}
KK^{M_g} (\mbox{Res}_{{V_g}}^{M_g} \, A_3, \mbox{Res}_{{V_g}}^{M_g} \, B) \cong KK^{{V_g}}(A_3, g g^* B)
\end{equation}
for every $V_g$-algebra $B$ by Lemma \ref{lemmaKKpL} and (\ref{equKKtildeG}). 
Hence, since $A_{3}$ satisfies the assumption,
the algebra 
$A_4:=\mbox{Res}_{V_g}^{M_g} \, A_3$ appearing in (\ref{successiveInd})
does it also by (\ref{eqMgpL}).
Successively we proceed in the same vein for the final three expressions
$\mbox{Ind}_{H_3}^G$, $\mbox{Res}_{L_g'}^{H_3}$ and $\mbox{Ind}_{M_g}^{L_g'}$
in (\ref{successiveInd}) until the assumption is verified for $A$.
The proof for arbitrary $A \in \calc \cali_0$ is analog.
%
%
%
%
%
\end{proof}



\section{The Baum--Connes map}

\label{sectionDirac}

In this section we shall switch from the restriction functors to the $\mbox{Right}$-functors
of Definition \ref{defRightadjointtoInd}.
We shall prove the existence of simplicial approximations and even a Dirac morphism
and a Baum--Connes map for all coefficient algebras under some unverified assumptions
which are motivated by the last section.
Excepting Definition \ref{defRightadjointtoInd},
Sections \ref{sectionLemmasResInd} and \ref{sectionMotivationDiracmorphism} will not be needed in this section beside motivation.
Because of Proposition \ref{propositiontriangulatedcategory} it is assumed that $G$ is $E$-continous.

Let $\calf_H$ be the set of Definition \ref{defRightadjointtoInd} or any other countable set of compact objects of $KK^H$;
by Definition \ref{defRightadjointtoInd} there exists a right adjoint functor $\mbox{Right}_G^H$ for $\mbox{Ind}_H^G$.

We are going to introduce analogous sets to $\calc \cali_1$ and $\calc \cali_0$ by replacing the $\mbox{Res}$-functors by the $\mbox{Right}$-functors.
As a motivation observe that in Corollary \ref{corollaryCountableKKGI0} we proved that every element of $\calc \cali_0$
is of the form $\mbox{Ind}_H^G(A)$ with $A \in \calf_H$.

\begin{definition}      \label{defcj}
{\rm
Let us be given a $G$-algebra $Z$.
Set
$$\calc \calj_1:= \{\mbox{Ind}_{H_n}^G \mbox{Right}_G^{H_n}
\ldots \mbox{Ind}_{H_1}^G \mbox{Right}_G^{H_1} (Z) \,| \, H_i \subseteq G \,\mbox{comp. sub. s.},
\, n \ge 1\}$$
and $\calc \calj_0$ the countable set of objects of the form $\mbox{Ind}_{H}^G A$,
where $H$ is a finite subinverse semigroup of $G$ and $A \in \calf_{H}$.
}
\end{definition}

The following corollary is a slight modification of Brown's representability theorem.


\begin{corollary}[Cf. Lemma 6.3 of \cite{meyernest}]   \label{corollaryBrownrepres}
%
%
Assume that for all $A \in \calc \calj_0$ $KK^G(A,B)$ is countable for all $B \in KK^G$
and commutes with countable direct sums in the variable $B$.
Then
for any object $B$ in $KK^G$ there exist an object $\tilde B$ in $\langle \calc \calj_0 \rangle$
and a morphism $f \in KK^G(\tilde B,B)$ such that $f_*:KK^G(A,\tilde B) \rightarrow KK^G(A,B)$
($f_*(x):= f \circ x$ for $x \in KK^G(A,\tilde B)$) is an isomorphism for all objects $A$ in $\langle \calc \calj_0 \rangle$.
\end{corollary}


\begin{definition}[Cf. Definition 4.1 of \cite{meyernest}]   \label{defCompactlyinduced}
{\rm
An object $A$ in $KK^G$ is called {\em compactly induced} 
if there exists an object $B$ in $KK^G$ and a compact subinverse semigroup $H \subseteq G$
such that $A$ is isomorphic to $\mbox{Ind}_H^G(B)$ in $KK^G$.
The full subcategory of $KK^G$ of compactly induced objects is denoted by $\calc \calj$.
}
\end{definition}


%

\begin{definition}[Cf. Definition 4.5 of \cite{meyernest}]  \label{defDiracMorphism}
{\rm
Let $Z$ be a $G$-algebra.
A {\em $\calc \calj$-simplicial approximation} for $Z$ 
is an element
$f \in KK^G(B,Z)$ for some object $B$ in $\langle \calc \calj \rangle$
such that $\mbox{Right}_G^H(f)$ is invertible in $KK^H$ for all compact subinverse semigroups
$H$ of $G$.
If $Z=C_0(X)$ we particularly call $f$ a {\em Dirac morphism}.
}
\end{definition}

As a motivation for the next proposition recall Remark \ref{remarkCI0} and notice that we have shown in identity
(\ref{successiveInd}) that every element of $\calc \cali_1$ is the countable
sum of elements in $\calc \cali_0$.

\begin{proposition}[Cf. Proposition 4.6 of \cite{meyernest}]     \label{propDiracMorphism}
%
%
Suppose that $G$ is $E$-continuous.
Assume that $KK^H(A,B)$ is countable and commutes with countable direct sums in $B$
for all finite subinverse semigroups $H$ of $G$ and $A \in \calf_H$ (see Definition \ref{defRightadjointtoInd}).
Assume that 
the $\mbox{Right}$-functors 
commute with countable direct sums.
%
%
Let $Z$ be a $G$-algebra.
Suppose that every object of $\calc \calj_1$ can be expressed 
as a countable direct sum of objects
of $\calc \calj_0$ up to $KK^G$-equivalence.
%
%
%
%
%

Assume that $\mbox{Right}_G^H(Z) \in \calf_H$ for all finite subsemigroups $H$ of $G$.
Then $Z$ has a $\calc \calj$-simplicial approximation.

%
\end{proposition}

\begin{proof}
Assume without loss of generality that for every $A \in \calf_H$, $\mbox{Ind}_H^G A \in \calc \calj_0$ 
appears as a summand 
of some $B \in \calc \calj_1$; if not so, 
simply restrict $\calf_H$ to a 
smaller set.
Notice that our assumptions imply the validity of the assumption of Corollary
\ref{corollaryBrownrepres}, see Remark \ref{remarkCI0}.

Apply Corollary \ref{corollaryBrownrepres} to $B:= Z$ and obtain an object $P \in \langle \calc \calj_0 \rangle \subseteq KK^G$
and a morphism $D \in KK^G(P,Z)$ (where $P:= \tilde B$ and $D:= f$ from Corollary \ref{corollaryBrownrepres})
such that
\begin{equation}   \label{diracinduced}
D_*: KK^G(A,P) \rightarrow KK^G(A,Z)
\end{equation}
is a group isomorphism for all $A \in \langle \calc \calj_0 \rangle$.
We want to show that $\mbox{Right}_G^H(D)$ is an isomorphism for every compact 
subinverse semigroup $H$ of $G$
(see Definition \ref{defDiracMorphism});
so fix any such $H$.
To this end it is sufficient to show that both induced group homomorphisms
$$\mbox{Right}_G^H(D)_*:KK^H(\mbox{Right}_G^H \,P,\mbox{Right}_G^H \,P) \rightarrow KK^H(\mbox{Right}_G^H \,P,\mbox{Right}_G^H \,Z)$$
and
$$
\mbox{Right}_G^H(D)_*:KK^H(\mbox{Right}_G^H \,Z, \mbox{Right}_G^H \,P) \rightarrow KK^H(\mbox{Right}_G^H \,Z, \mbox{Right}_G^H \,Z)
$$
are isomorphisms.
For verifying that the first stated $\mbox{Right}_G^H(D)_*$ is an isomorphism
it is sufficient to show that
\begin{equation}  \label{diracinduced2}
\mbox{Right}_G^H(D)_*:KK^H(\mbox{Right}_G^H \,A, \mbox{Right}_G^H \,P) \rightarrow KK^H(\mbox{Right}_G^H \, A, \mbox{Right}_G^H \,Z)
\end{equation}
is an isomorphism for all $A \in \calc \calj_0$ because
$P \in \langle \calc \calj_0 \rangle$.

We consider first the case that $A \in \calc \calj_1$. 
Applying on both ends of (\ref{diracinduced2})
the adjointness relation between $\mbox{Ind}$ and $\mbox{Right}$,
(\ref{diracinduced2})
turns to
\begin{equation}  \label{diracinduced3}
D_* : KK^G(\mbox{Ind}_H^G \mbox{Right}_G^H \,A, P) \rightarrow KK^G(\mbox{Ind}_H^G \mbox{Right}_G^H \, A, Z).
\end{equation}
But since $\mbox{Ind}_H^G \mbox{Right}_G^H \,A$ is in $\calc \calj_1$, 
and hence a
countable direct sum of objects in $\calc \calj_0$
by assumption,
$\mbox{Ind}_H^G \mbox{Right}_G^H \,A$ is also in $\langle \calc \calj_0 \rangle$ by Definitions \ref{defLocSubcat} and \ref{defGeneratedLocCat},
and hence (\ref{diracinduced3}) and so (\ref{diracinduced2}) are isomorphisms by (\ref{diracinduced}).
%

We may write $A \cong \bigoplus_j B_j$
$KK^G$-equivalently by assumption, where $B_j \in \calc \calj_0$.
The canonical injection and projection $\mbox{Right}_G^H \,B_j \stackrel{p}{\rightarrow} \mbox{Right}^H_G \,A
\stackrel{f}{\rightarrow} \mbox{Right}^H_G \, B_j$ to the $j$th coordinate satisfy $id =(f p)^* = p^* f^*$, and an 
analog diagram
as in (\ref{commdiagram}) shows that the isomorphism (\ref{diracinduced2}) is also an isomorphism for $A:=B_j$.
By varying over all $A \in \calc \calj_1$ and all coordinate projections $j$,
we see
that (\ref{diracinduced2}) is an isomorphism for all $A \in \calc \calj_0$.

%
That the second homomorphism $\mbox{Right}_G^H(D)_*$ is an isomorphism follows from
(\ref{diracinduced}) applied to $A:=\mbox{Ind}_H^G \mbox{Right}_G^H \,Z \in
\calc \calj_0$.
%
%
\end{proof}

The last proposition might offer a chance for defining a Baum--Connes map:

\begin{remark}    \label{remarkBC}
{\rm
In \cite{burgiNoteBaumConnes} we have defined a so-called fibered restriction functor $\mbox{R}_G^H:KK^G \rightarrow KK^H$
which is right adjoint to the induction functor on the subclass of fibered $G$-algebras of $KK^G$.
If $\calf_H$ should cover all finite dimensional commutative $H$-algebras (which appears natural)
and $\mbox{Right}_G^H$ should coincide with $\mbox{R}_G^H$  
then evidently $\mbox{Right}_G^H(Z) \in \calf_H$ for $Z=C_0(X)$.
If the other natural assumptions of Proposition \ref{propDiracMorphism} should hold true
then its application to $Z=C_0(X)$ would
yield a $\calc \calj$-simplicial approximation  and thus a Baum--Connes
map for all coefficient algebras $A$
(by tensoring a $\calc \calj$-simplicial approximation $D$ 
for $C_0(X)$ with
$A$, that is, forming $D \otimes^{C_0(X)} A$); see
\cite{meyernest} or \cite[Section 10]{burgiNoteBaumConnes} for the concept.
%
%
The {\em Baum--Connes map} with coefficient algebra $A$ would then be defined to be the homomorphism $K(B \widehat \rtimes G) \rightarrow
K(A \widehat \rtimes G)$ (Sieben's crossed product) induced by taking the Kasparov product with the element $\widehat{j^G}(D)
\in KK(B \widehat \rtimes G,A \widehat \rtimes G)$ (descent homomorphism) for any $\calc \calj$-simplicial
approximation $D \in KK(B,A)$ of $A$.
}
\end{remark}

\bibliographystyle{plain}
\bibliography{references}

\end{document}